\documentclass[10pt]{article}
\usepackage{hyperref}
\usepackage{amssymb}
\usepackage{amsmath}
\input{liemacs10.sty} 
\usepackage{color}

\newcommand{\regexp}{\mathop{{\rm regexp}}\nolimits}

\newcommand{\Cont}{\mathop{{\rm Cont}}\nolimits}

\newcommand{\rank}{\mathop{{\rm rank}}\nolimits}

\newcommand{\comp}{\mathop{{\rm comp}}\nolimits}

\newcommand{\Conj}{\mathop{{\rm Conj}}\nolimits}

\renewcommand{\phi}{\varphi}

\renewcommand\mlabel{\label} 

\begin{document}

\title{Elliptic domains in Lie groups}
\author{ Jakob Hedicke\begin{footnote}
    {Centre de recherches mathématiques (CRM), Université de Montréal, 2920 Chemin de la tour, Montréal (Québec), H3T 1J4, Canada; jakob.hedicke@gmail.com  }\end{footnote}\, and Karl-Hermann Neeb
  \begin{footnote}
    {Department Mathematik, Friedrich-Alexander-Universit\"at Erlangen-N\"urnberg, Cauerstrasse 11, 91058 Erlangen, Germany; neeb@math.fau.de  }\end{footnote}
}
\date{}

\maketitle

\begin{abstract} An element $g$ of a Lie group
  is called stably elliptic if it is contained in the interior
  of the set $G^e$ of elliptic elements, characterized by
  the property that $\Ad(g)$ generates a relatively compact subgroup. 
  Stably elliptic elements appear naturally in the geometry of
  causal symmetric spaces and in representation theory. 
  We characterize stably elliptic elements in terms
  of the fixed point algebra of 
  $\Ad(g)$ and show that the connected components of the set
  $G^{se}$ of stably elliptic elements can be described in terms
  of the Weyl group action on a compactly embedded Cartan subalgebra.

In the case of simple hermitian Lie groups we relate stably elliptic elements to maximal invariant cones and the associated subsemigroups.
In particular we show that the basic connected component $G^{se}(0)$
can be characterized in terms of the compactness of order intervals 
and that $G^{se}(0)$ is globally hyperbolic with respect to
the induced biinvariant causal structure.
\end{abstract}

\tableofcontents

\section{Introduction}

Let $G$ be a connected Lie group with Lie algebra~$\g$.
We call an element $g \in G$ {\it elliptic}, or {\it compactly embedded}, if
the closed subgroup of $\Aut(\g)$ generated by $\Ad(g)$ is compact.
The set of elliptic elements is denoted $G^e$ and its interior,
the set of {\it stably elliptic elements}, by $G^{se}$.
We show that $G^{se} \not=\eset$ if and only if
$\g$ possesses a Cartan subalgebra
which is compactly embedded in the sense that
its adjoint image generates a relatively compact group.
This can be checked easily. We study the structure of the
connected components of $G^{se}$, the so-called {\it elliptic domains}.
Our approach is based on 
Theorem~\ref{thm:compemb}, which characterizes stably elliptic elements
$g \in G$ as those for which the fixed point algebra
$\g^{\Ad(g)} = \Fix(\Ad(g))$ is  compactly embedded. 
As for the enumeration of the connected components of $G^{se}$,
we show that they are in one-to-one correspondence
with Weyl group orbits of connected components of
$T^{se} := T \cap G^{se}$, where $T = \exp \ft$ for a
compactly embedded Cartan subalgebra (Theorem~\ref{thm:conncomp}).

For simple hermitian Lie algebras $\g$, we connect in
Section~\ref{sec:4} these concepts to
maximal invariant cones $W$ and the associated subsemigroups
$S_W = \oline{\la \exp W \ra}$. We show that
the basic connected component $G^{se}(0)$
is contained in $S_W$ and can be characterized
in terms of compactness of order intervals if $S_W \not=G$,
but in all cases it carries a biinvariant causal structure
for which it is globally hyperbolic, i.e., the corresponding
causal intervals are compact.

We expect these results on elliptic domains in general 
Lie groups to  be useful in a variety of contexts.
The present investigation was mainly motivated by
the theory of causal symmetric spaces and their
connections with crown domains (\cite{GK02}, \cite[\S 6.1]{NO23}).
In this context the duality between a causal symmetric space 
of the form $G_\C/G$ and the group $G$, endowed with a biinvariant
causal structure, naturally lead to elliptic domains in $G$
that arise as intersections of complex tube domains 
in $G_\C$ with the real group $G$.
It also provides a tool to study
domains in homogeneous spaces for which all
stabilizer Lie algebras are compactly embedded (see \cite{GK02}
fore related conditions in the context of crown domains).
For the conjugation action of $G$ on itself, 
Theorem~\ref{thm:compemb} characterizes elements in $G^{se}$
as those with compactly embedded stabilizer Lie algebras.

In the representation theory  of semisimple Lie groups,
elliptic domains appear naturally in the theory of discrete
series representations and their characters \cite{HC65}.
Here the condition $G^{se} \not=\eset$  is equivalent to the
existence of (relative) discrete series representations.
Stably elliptic elements in the linear symplectic group are of importance in the study of periodic Hamiltonian systems and their stability and can be characterised in terms of Krein theory, see e.g.~\cite{Ek89}.

We now describe our results and the structure of this paper in some
more detail. In Section~\ref{sec:1} we collect some observations on the
set $G^e$ of elliptic elements of a connected Lie group~$G$.
We are mostly interested in its interior $G^{se}$, the set of
stably elliptic elements. In Proposition~\ref{prop:c.1ax} we
show that this set is non-empty if and only if
$\g$ contains a compactly embedded
Cartan subalgebra. A key result is Theorem~\ref{thm:compemb}
which characterizes stably elliptic elements as those
elliptic elements whose centralizer is contained in~$G^e$. 
This is complemented by Theorem~\ref{thm:compchar}
which characterizes compact elements in $\g$ as those for which
the corresponding conjugation flow on the simply
connected group $\tilde G$ with Lie algebra $\g$
has only relatively compact orbits.  We also obtain a
natural polar decomposition of $G^{se} \cong \R^N \times K^{se}$
(Proposition~\ref{prop:2.11})
reducing many questions to the subset $K^{se} := K \cap G^{se}$,
where $K = \la \exp \fk\ra$ for a maximal compactly embedded
subalgebra $\fk$. 

In Section~\ref{sec:2} we address the description of the
connected components of the open subset~$G^{se}$ 
of stably elliptic elements.
As $G$ is connected, all connected components are
invariant under conjugation. Using the fact that
$\g$ has a compactly embedded Cartan subalgebra $\ft$
(Proposition~\ref{prop:c.1ax}), we can use
the root decomposition with respect to $\ft$ to 
label connected components of  $G^{se}$ by $\Z$-valued functions
on the set of non-compact roots (Theorem~\ref{thm:conncomp}).

In Section~\ref{sec:3} we specialize to the case where $\g$ is a simple 
real Lie algebra which is {\it hermitian}, i.e.,
any  maximal compactly embedded subalgebra $\fk$ has non-zero
center $\fz(\fk)$ (\cite{Hel78}, \cite{Ne99}). 
Here we are mostly interested in the ``basic'' connected component
$G^{se}(0)$. Proposition~\ref{prop:3.11} shows that
this domain ``does not depend on $G$'', in the sense that
$\exp \: \g^{se}(0) \to G^{se}(0)$ is always a diffeomorphism. 
Further, the Properness Theorem~\ref{properness:theorem}
provides important information
on the polar decomposition of $G^{se}(0)$ 
that will be used in Section~\ref{sec:6} below. 

In Section~\ref{sec:4} we turn to the relation
between   elliptic domains with invariant cones in a
  simple hermitian Lie algebra $\g$. 
  Recall that a
  simple real Lie algebra $\g$ possesses a pointed generating closed
  convex cone $W \subeq \g$ invariant under $\Ad(G)$ if and only
  if $\g$ is hermitian (\cite{HHL89, Ne99}). In this case there exists a
  maximal such cone $W^{\rm max}_\g$, unique up to sign, with the
  property that any other pointed generating invariant cone $W$
  is either contained in $ W^{\rm max}_\g$ or $- W^{\rm max}_\g$. Here
  we are interested only in the maximal cone $W := W_\g^{\rm max}$
  and the corresponding
  biinvariant causal structure on $G$, defined by the cone field
  $(g.W)_{g \in G}$. If $G$ is simply connected,
  then this causal structure corresponds to a global order on
  $G$ if and only if $\g$ is of {\it tube type},
    i.e., if the corresponding Riemannian symmetric 
space $G/K$ is biholomorphic to a tube domain $T_\Omega = \Omega + i V$, 
where $\Omega$ is an open symmetric cone in the real vector
space~$V$ 
(\cite{Ol82}, \cite{Ne90}, \cite[Thm.~VIII.12]{Ne93}).
Then $S_W = \oline{\la \exp W \ra} \subeq G$ is a proper closed
subsemigroup. 
The main result of this section is Theorem~\ref{thm:4.x} which
characterizes $G^{se}(0)$ as the interior of the set
$\comp(S_W)$ of all those elements $s \in S_W$ for which the order interval
$S_W \cap s S_W^{-1}$ is compact. As a consequence, we obtain
that the causal structure on $G^{se}(0)$ is globally
hyperbolic (Corollary~\ref{cor:globhyp1}).

In Section~\ref{sec:6} we show with different methods
that the biinvariant causal
structure on $G^{se}(0)$ is also globally hyperbolic
if the hermitian Lie algebra $\g$ is not of tube type. 
Then the subsemigroup $S_W$ coincides with the
whole group, so that we cannot refer to the global
order structure corresponding to the biinvariant
cone field $(g.W)_{g \in G}$. The methods we use in this case
have been adapted from \cite{He22}, where a similar result was shown
for the symplectic group $G = \Sp_{2n}(\R)$.

\vspace{5mm}

\nin {\bf Acknowledgment:} We thank Tobias Hartnick
for helping us with some background in quasimorphisms.
We also thank Alain Valette for helpful discussions
concerning Example~2.12. J.H. is supported by funding from the Fondation Courtois and  K.-H.N. acknowledges support by DFG-grant NE 413/10-2.\\

\nin {\bf Notation:} 
\begin{itemize}
\item $Z_G(h) = \{ g \in G \: g h = h g\}$ denotes the centralizer
  of $h \in G$ in the group $G$. For subsets $S \subeq G$ we write
  $Z_G(S) := \bigcap_{h \in S} Z_G(h)$.
  The normalizer of $S$ is $N_G(S) = \{ g \in G \: g S g^{-1} = S\}$. 
\item $Z_G(x) = \{ g \in G \: \Ad(g)x = x\}$ denotes the centralizer 
  of $x \in \g$ in the group $G$ and likewise
  $Z_G(S)$ for subsets $S \subeq \g$.
\item For a subset $S \subeq \g$ we write
  $\fz_\g(S) := \{ x \in \g \:  [x,S] = \{0\}\}$ for its
  centralizer in $\g$. 
\item $e$ denotes the neutral element in $G$ and
  $G_e$ the identity component of the Lie group $G$.
\item $\L(G)$ or $\g$ denotes the Lie algebra of $G$.   
\item $\fk$ will always denote a maximal compactly embedded subalgebra
  of the Lie algebra $\g$ and $\ft \subeq \fk$ a maximal abelian subalgebra,
  so that $\ft$ is maximal abelian compactly embedded in $\g$.
   Further, $\fp$ denotes a $\fk$-invariant complement. 
\item For a Lie algebra $\g$, we write
  $\Inn(\g) = \la e^{\ad \g} \ra \subeq \Aut(\g)$ for the group
  of inner automorphisms. 
\end{itemize}

\section{Elliptic elements of Lie groups} 
\mlabel{sec:1}

In this section we take a closer look at the
set $G^e$ of elliptic elements of a connected Lie group~$G$.
We are mostly interesting in its interior $G^{se}$, the set of
stably elliptic elements. In Proposition~\ref{prop:c.1ax} we
shall see that this set is non-empty if and only if
the Lie algebra $\g$ contains a compactly embedded
Cartan subalgebra. A key result is Theorem~\ref{thm:compemb}
which characterizes stably elliptic elements as those
elliptic elements whose centralizer consists of elliptic elements.
This is complemented by Theorem~\ref{thm:compchar}
that characterizes compact elements in $\g$ as those for which
the corresponding conjugation flow on the simply
connected group $\tilde G$ with Lie algebra $\g$
has compact orbit closures.
It is also interesting to observe that the set $G^{se}$
has a natural polar decomposition $G^{se} \cong  \R^N \times K^{se}$
(Proposition~\ref{prop:2.11})
reducting many questions to the subset $K^{se} = K \cap G^{se}$,
where $K = \la \exp \fk\ra$ for a maximal compactly embedded
subalgebra $\fk$.

\begin{defn} \mlabel{def:1.1} (cf.\ \cite{HH89})
  (a) We call an element $x$ of a Lie algebra $\g$
  {\it compact/elliptic}
  if $\oline{\exp(\R \ad x)}$ is a compact subgroup of $\Aut(\g)$, i.e.,
if $\ad x$ is semisimple with purely imaginary spectrum. 
A~subalgebra $\fh \subeq \g$ is called {\it compactly embedded} if
$\oline{\la \exp(\ad \fh) \ra}$ is a compact subgroup of $\Aut(\g)$.
By \cite[Lemma~VII.1.5]{Ne99}, this is equivalent to
$\fh \subeq \comp(\g)$. 

\nin (b) Let $G$ be a connected Lie group. 
We call an element $g \in G$ {\it compact} if the cyclic subgroup 
$g^\Z \subeq G$ is relatively compact.

We likewise call $g \in G$ {\it compactly embedded} or
{\it elliptic} if $\Ad(g)$ is compact 
in $\Aut(\g)$, and write~$G^e$ for the set of
elliptic elements. Elements in the  interior of this set are
called {\it stably elliptic} and we write
\begin{equation}
  \label{eq:gse}
  G^{se} := (G^e)^\circ
\end{equation}
for this set.
We also put
\begin{equation}
  \label{eq:gse2}
  \g^{se} :=  \exp^{-1}(G^{se})
\end{equation}
and note that $\g^{se}$ is open in $\g$ because $\exp$ is continuous.
\end{defn}

\begin{rem} To justify the notation $\g^{se}$, we have to verify
  that this set does not depend on the Lie group $G$.
  So let $q_G \:  \tilde G \to G$  denote the universal covering morphism.
  Then the definition of $G^e$ implies that
  $\tilde G^e = q_G^{-1}(G^e)$, so that 
  \[ \tilde G^{se} =  q_G^{-1}(G^{se}).\]
  This in turn shows that, for $x \in \g$, the conditions
  $\exp_{\tilde G}(x) \in \tilde G^{se}$ and 
  $\exp_G(x) \in G^{se}$ are equivalent, and thus $\g^{se}$ is independent of
  $G$, because all connected Lie groups with the same Lie algebra have
  isomorphic universal covering groups.  
\end{rem}

\subsection{Existence of stably elliptic elements} 

\begin{lem} \mlabel{lem:c.1a} 
Let $K \subeq G$ be the integral subgroup corresponding to a 
maximal compactly embedded subalgebra $\fk \subeq \g$.
Then the  following assertions hold: 
\begin{itemize}
\item[\rm(i)] $G^e Z(G) = G^e$. 
\item[\rm(ii)] $Z_G(zg) = Z_G(g)$ for $g \in G, z \in Z(G)$. 
\item[\rm(iii)] $G^e   = \bigcup_{g \in G} g K g^{-1}$. 
\item[\rm(iv)] $G^e = \exp(\comp(\g))$ and
  $\comp(\g) = \exp^{-1}(G^e).$ 
\item[\rm(v)]  $\g^{se} \subeq \comp(\g)^\circ
  = \{ x \in \comp(\g) \: \ker(\ad x)\ \mbox{ compactly embedded}\}$. 
\item[\rm(vi)] $G^{se} = \exp(\g^{se})$.
\end{itemize}
\end{lem}

\begin{prf} (i) and (ii) are trivial. 

\nin (iii) Clearly, $G^e$ is conjugation invariant and contains $K$, 
hence all conjugates $gKg^{-1}$ for $g \in G$. 

Suppose, conversely, that $g \in G^e$. 
Then $\Ad(g)^\Z$ is contained in a maximal compact subgroup
$U \subeq \oline{\Ad(G)}$. Since
$\oline{\Ad(K)}$ is maximal compact in $\oline{\Ad(G)}$,
maximal compact subgroups are conjugate under inner automorphisms,
and $\oline{\Ad(G)} = \Ad(G) \oline{\Ad(K)}$
(\cite[Thm.~VII.1.4(iii),(iv)]{Ne99}, \cite{HH89}), we infer that there
exists an element $h \in G$ with
$hgh^{-1} \in \Ad^{-1}(\oline{\Ad(K)}) = K$
(see \cite[Thm.~VII.1.4(iii)]{Ne99}). This proves (iii).

\nin (iv) The inclusion
$\exp(\comp(\g)) \subeq G^e$ follows directly from the
definition. As the exponential function of $K$ is surjective
and $\comp(\g) =  \Ad(G)\fk$ (\cite[Thm.~VII.1.4]{Ne99}), we derive from (iii) that $G^e \subeq \exp(\comp(\g))$.

It remains to show that
$\exp x \in G^e$ implies that $x \in \comp(\g)$.
This follows from the compactness of the set 
\[ \oline{\Ad(\exp \R x)} 
\subeq \Ad(\exp [0,1]x) \oline{\Ad(\exp \Z x)},  \]
which is a product of two compact sets.

\nin (v) From (iv) we know that $\comp(\g) = \exp^{-1}(G^e)$,
so that the continuity of $\exp$ implies that
$\g^{se} = \exp^{-1}((G^e)^\circ) \subeq \comp(\g)^\circ$.
The equality follows from \cite[Lemma~VII.1.7]{Ne99}.

\nin (vi) As (iv) implies that $G^{se} \subeq \exp(\g)$, we have
  $G^{se} = \exp(\exp^{-1}(G^{se})) = \exp(\g^{se})$.
\end{prf}

\begin{prop} \mlabel{prop:c.1ax}  
  The following are equivalent:
  \begin{itemize}
  \item[\rm(a)] $G^{se} \not=\eset$.
  \item[\rm(b)]   $\g^{se} \not=\eset$. 
  \item[\rm(c)] $\comp(\g)$ has interior points.
  \item[\rm(d)] $\g$ contains a compactly embedded Cartan subalgebra.
  \end{itemize}
\end{prop}

\begin{prf} (a) $\Leftrightarrow$ (b): Lemma~\ref{lem:c.1a}(iv) 
shows that $G^{se} \not=\eset$ implies that
$\g^{se} = \exp^{-1}(G^{se}) \not=\eset$
and the converse follows directly from the definition.

\nin   (b) $\Rarrow$ (c): From Lemma~\ref{lem:c.1a}(v)
  we know that, if $\g^{se} \not=\eset$, then
  $\comp(\g)$ has interior points.

\nin (c) $\Leftrightarrow$ (d):   
By \cite[Thm.~VII.1.8(i)]{Ne99}, (c) is is equivalent to the existence of a
compactly embedded Cartan subalgebra.

\nin (d) $\Rarrow$ (b): 
Let $U \subeq \g$ be an open $0$-neighborhood for which
$\exp\res_U \: U \to \exp(U)$ is a chart of $G$ and
let $\ft \subeq \g$ be a compactly embedded Cartan subalgebra. 
Then $U \cap \ft$ contains a regular element $x$, i.e.,
$\ker(\ad x) = \ft$. Then  \cite[Lemma~VII.1.7]{Ne99}
(or \cite[Thm.~2.9]{HH89}) implies
that $x \in \comp(\g)^\circ$. Thus
$\exp(U \cap \comp(\g)^\circ) \subeq G^e$ is an open neighborhood
of $\exp x$, and therefore $x \in \g^{se}$.
\end{prf}

\subsection{Charactization of stably elliptic elements} 

\begin{lem} \mlabel{lem:unique-max}
If  $\fq \subeq\g$ is a compactly embedded subalgebra  
 of full rank in $\g$, then $\fq$ is contained in a unique maximal
  compactly embedded subalgebra.   
\end{lem}

\begin{prf} By assumption, $\fq$ contains a Cartan subalgebra
  $\ft$ of $\fq$ and then $\ft$ is compactly embedded.
  Let $\fk\supeq \fq$ be a maximal compactly embedded subalgebra of $\g$.
  Then $\fk$ contains $\ft$, hence is unique by
  \cite[Prop.~VII.2.5]{Ne99}.   
\end{prf}

\begin{lem} \mlabel{lem:unique-with-t} 
  Let $\ft \subeq\g$ be a compactly embedded Cartan subalgebra
  and $\fk \supeq \ft$ be maximal compactly embedded.
  Then a Lie subalgebra $\fq \supeq \ft$ is compactly
  embedded if and only if $\fq \subeq \fk$.
\end{lem}

\begin{prf} Any subalgebra of $\fk$ is compactly embedded.
  If, conversely, $\fq \supeq \ft$ is compactly embedded,
  then it is contained in a maximal compactly embedded subalgebra
  $\fk_1 \supeq \fq \supeq \ft$, and Lemma~\ref{lem:unique-max}
  implies $\fk_1 = \fk$.
\end{prf}

\begin{lem} \mlabel{lem:cent-in-k} If $x \in \fk$
    and $\Fix(e^{\ad x})$ is compactly embedded, then
    $\Fix(e^{\ad x}) \subeq \fk$.
\end{lem}

\begin{prf} First, $\Fix(e^{\ad x}) \supeq \ker(\ad x)$
  implies that $\ker(\ad x) = \fz_\g(x)$ is compactly embedded.
  Since $\ad x$ is semisimple,
this subspace contains a Cartan subalgebra of~$\g$.
So $\g$ possesses a compactly embedded Cartan subalgebra
  and thus every Cartan subalgebra 
  $\ft \subeq \fz_\fk(x)$ is a compactly embedded Cartan
  subalgebra of $\g$.
  In view of   Lemma~\ref{lem:unique-with-t}, the inclusions
  $\ft \subeq \fz_\g(x)  \subeq \Fix(e^{\ad x})$ 
  imply that $\Fix(e^{\ad x}) \subeq \fk$.
\end{prf}

\begin{lem} \mlabel{lem:compact-central}
Let $G$ be a connected Lie group,
$\fk \subeq \g$ a maximal compactly embedded Lie subalgebra
and $K = \exp \fk$.
If $k \in K$ is such that $\g^{\Ad(k)} = \Fix(\Ad(k))$
is compactly embedded, then
\begin{equation}
  \label{eq:fix0}
  gkg^{-1} \in K \qquad \Rarrow \qquad g \in K
\end{equation}
and in particular
\begin{equation}
  \label{eq:fix1}
  Z_G(k) = Z_K(k)\subeq K.
\end{equation}
\end{lem}

\begin{prf}   We write $k = \exp x$ with some $x \in \fk$ and observe that
  \begin{equation}
    \label{eq:infk}
      \Fix(\Ad(k)) = \Fix(e^{\ad x}) \subeq \fk
  \end{equation}
follows from Lemma~\ref{lem:cent-in-k}.

By \cite[Thm.~14.3.7]{HN12}
  there exist three $\Ad(K)$-invariant linear subspaces 
  $E_1, E_2, E_3 \subeq \g$ (possibly trivial) such that the map
  \[ \Phi \: E_1 \times E_2 \times E_3 \times K \to G,\quad
  \Phi(x_1, x_2, x_3, k) := \exp(x_1) \exp(x_2) \exp(x_3) k \]
  is a $K$-equivariant diffeomorphism for the $K$ action on
  $G$ by conjugation and on the left by
  \[ q.(x_1, x_2, x_3, k) := (\Ad(q)x_1, \Ad(q)x_2, \Ad(q)x_3,
    qkq^{-1})\quad \mbox{ for } \quad q \in K.\]
  For $g = \Phi(x_1, x_2, x_3,k_0)$ we consider the condition
  \[ g kg^{-1}
    = \Phi(x_1, x_2, x_3,k_0) k \Phi(x_1, x_2, x_3,k_0)^{-1} 
    = \Phi(x_1, x_2, x_3,e) k_0 k k_0^{-1} \Phi(x_1, x_2, x_3,e)^{-1} \in K.\]
  As $\Fix(\Ad(k_0 k k_0^{-1})) = \Ad(k_0) \Fix(\Ad(k))$ is compactly embedded,
  we may assume that $k_0 = e$. We have to show that $x_j = 0$ for $j = 1,2,3$.
  The relation $k' := g kg^{-1} \in K$ implies that
  \[  \Phi(x_1, x_2, x_3,k)  = gk = k'g
    = \Phi(\Ad(k')x_1, \Ad(k')x_2, \Ad(k')x_3,k'),\]
  so that we obtain $k = k'$ and $\Ad(k)x_j = x_j$ for $j = 1,2,3$.
  With \eqref{eq:infk}, we finally obtain
\[ x_j \in E_j \cap \Fix(\Ad(k)) \subeq E_j \cap \fk = \{0\}.\qedhere\]
\end{prf}

The following theorem generalizes the characterization
of interior elements of $\comp(\g)$ as those whose
centralizer is compactly embedded 
(\cite[Thm.~2.9]{HH89}, \cite[Lemma~VII.1.7]{Ne99})
from the Lie algebra level to the group context.
It is a key result of this paper.  

\begin{thm} \mlabel{thm:compemb}
  {\rm(Characterization of stably elliptic elements)} 
For an elliptic element 
$g \in G^e$, the following are equivalent: 
\begin{itemize}
\item[\rm(a)] $g$ is stably elliptic, i.e.,
  $g \in G^{\rm se} = (G^e)^\circ$. 
\item[\rm(b)] $\L(Z_G(g)) = \g^{\Ad(g)} = \Fix(\Ad(g))$ is compactly embedded. 
\item[\rm(c)] $Z_G(g) \subeq G^e$. 
\end{itemize}
\end{thm}

\begin{prf} In view of Lemma~\ref{lem:c.1a}(iv), there exists 
an element $x \in \g$ with $g = \exp x$. 

\nin (a) $\Rarrow$ (b): Let $g \in \comp_{\Ad(g)}(G)^\circ$. 
Let $y \in \g^{\Ad(g)}$, so that $g$ commutes with the one-parameter 
subgroup $\exp(\R y)$. Our assumption implies the existence of an 
$\eps > 0$ with $g \exp(t y) \in G^e$ for $|t| < \eps$. 
As $g$ and $\exp(ty)$ commute, it follows that $\exp(ty) \in G^e$, 
and thus $y \in \comp(\g)$ by Lemma~\ref{lem:c.1a}(iv). This shows that 
$\g^{\Ad(g)} \subeq \comp(\g)$, so that $\g^{\Ad(g)}$ is compactly 
embedded by \cite[Lemma~VII.1.5(b)]{Ne99}
(cf.\ also \cite[Cor.~2.6]{HH89}). 

\nin (b) $\Rarrow$ (c):
Suppose that $\g^{\Ad(g)}$ is compactly embedded.
We apply Lemma~\ref{lem:compact-central}
with a maximal compactly embedded subalgebra $\fk$ containing~$x$
and obtain for the corresponding integral subgroup
$K = \exp \fk$ that
\[ Z_G(g) = Z_K(g)  \subeq K \subeq G^e.\] 
     
\nin (c) $\Rarrow$ (b) follows from $\g^{\Ad(g)}= \L(Z_G(g))$
because (c) entails that $\L(Z_G(g)) \subeq \comp(\g)$ by
Lemma~\ref{lem:c.1a}(iv).

\nin (b) $\Rarrow$ (a): Assume that $\g^{\Ad(g)}$ is compactly embedded,
hence
contained in a maximal compactly embedded subalgebra 
$\fk \subeq \g$. Then $x \in \g^{\Ad(g)} \subeq \fk$ and 
$g = \exp x \in K := \exp \fk$.
Let $\fp \subeq \g$ be a $\fk$-invariant vector space
complement of $\fk$. 
We consider the map 
\[ \Phi \: \fp \times K\to \comp_{\Ad(G)} \subeq G, \quad 
\Phi(y,k) := (\exp y) k (\exp y)^{-1}.\] 
Then 
\[ \im(T_{(0,k)}(\Phi)) 
= k.\big(\fk + (\Ad(k)^{-1}-\1) \fp\big) \] 
coincides with $T_k(G)$ if and only if 
$\fp^{\Ad(k)} = \{0\}$, which follows from $\g^{\Ad(k)} \subeq \fk$.
Now the Inverse Function Theorem 
implies that $g = \Phi(0,g) \in \im(\Phi)^\circ$ is an inner point of 
$G^e$.
\end{prf}

\begin{prob}
  Show that, if $G$ is a connected Lie group
  and the subalgebra
  \[ \Fix(\Ad(g)) = \ker(\Ad(g) - \1) \]
  is compactly embedded in $\g$, then $g$ is elliptic.
    If $g = \exp x$ for some $x \in \g$, then $x \in \Fix(\Ad(g))
  \subeq \comp(\g)$, so that Lemma~\ref{lem:c.1a} implies that
  $g \in G^e$. If $g \not\in \exp(\g)$, we do not see how to argue. 
\end{prob}

\begin{rem} \mlabel{rem:1.7} Let $k = \exp x$. We have 
\begin{equation}
  \label{eq:cent1}
  Z_G(\exp x)
  = \{ g \in G \: \exp x g \exp(-x) = g \}
  = \{ g \in G \: \exp x = \exp(\Ad(g)x) \}
\end{equation}
and, for $x \in \comp(\g)$,  
\begin{equation}
  \label{eq:cent2}
  Z_\g(\exp x) = \Fix(e^{\ad x}) 
  = \bigoplus_{n\in \Z} \ker((\ad x)^2 + (2\pi n)^2\1).
\end{equation}
If $x$ is sufficiently small, so that the spectral radius is
less than $2\pi$, it follows that 
\[ \Fix(e^{\ad x}) = \ker(\ad x).\] 
\end{rem}

\begin{ex}
  The element
  \[ z := \pi \pmat{ 0 & -1 \\ 1 & 0} \in \g := \fsl_2(\R) \]
  is compact with $\ker(\ad z) = \so_2(\R) \subeq \comp(\g)$,
  so that $z \in \comp(\g)^\circ$. However,
  $\exp z = -\1 \in Z(G)$ and
  \[ Z(G) \cap G^{se} = \eset \]
  follows from $\g^{\Ad(g)} = \g$ for $g \in Z(G)$ (Theorem~\ref{thm:compemb}).   So $z \not\in \g^{se}$. 
\end{ex}

A Lie algebra $\g$ is called {\it quasihermitian}
if for one (and hence for all) maximal compactly embedded subalgebras
$\fk \subeq \g$ the center $\fz(\fk)$ satisfies
\[ \fk = \fz_\g(\z(\fk)) = \{ x \in \g \: [x,\fz(\fk)] =  \{0\}\}.\]
This implies that every Cartan subalgebra $\ft \subeq \fk$
is a compactly embedded Cartan subalgebra of $\g$
(\cite[Thm.~4.16]{HH89}, \cite{Ne99}). 

The following lemma translates the characterization of quasihermitian
Lie algebras to the group level. 

\begin{lem} \mlabel{lem:qh}
  If $\g$ is quasihermitian,
  then there exists an element $z \in Z(K)_e$ with $Z_G(z) = K$.
\end{lem}

\begin{prf} As $\g$ is quasihermitian,
  $\fz(\fk)$ contains  elements $y$ with $\fz_\g(y) = \fk$.  
  \begin{footnote}{Then $\g$ contains a compactly embedded Cartan subalgebra
      $\ft$ and no non-compact root vanishes on $\fz(\fk)$. Hence
      there exists a single element on which no non-compact root vanishes.}    
  \end{footnote}
  Passing to a suitable multiple, we may assume that the
  spectral radius of $y$ is less than $2\pi$.
  Then $\g^{\Ad(\exp y)} = \ker(\ad y) \subeq \fk$
  (cf.\ Remark~\ref{rem:1.7}).
  Now Lemma~\ref{lem:compact-central} implies that
  $Z_G(\exp y) \subeq K$, so that $K = Z_G(\exp y)$ because
  $\exp y$ is central in~$K$.  
\end{prf}

\begin{ex}
    The Lie algebra $\g = \so_{4,4}(\R)$ is simple and split of type $D_4$.
    The subalgebra $\fk = \so_4(\R) \oplus \so_4(\R)$ is maximal compactly
    embedded and
    $\ft \cong \so_2(\R)^{\oplus 4}$ is a compactly embedded Cartan subalgebra.
    In particular, $\g^{se} \not=\eset$ by Proposition~\ref{prop:c.1ax}.
    As $\fz(\fk) = \{0\}$, the Lie algebra $\g$ is not quasihermitian.  
\end{ex}

\subsection{Polar decomposition of elliptic domains} 

\begin{prop} \mlabel{prop:2.11} {\rm(Polar decomposition of
    elliptic domains)} 
  Let $G$ be a connected Lie group for which $G^{se} \not=\eset$
  and $K$ as above. Then there exist
  $\Ad(K)$-invariant subspace $E_1, E_2$ and $E_3$ in $\g$,
  for which 
  \[ \Phi\: E_1 \times E_2 \times E_3 \times K \to G,
    \quad (x_1, x_2, x_3,k) \mapsto
    \exp(x_1) \exp(x_2) \exp(x_3) k \]
  is a diffeomorphism and for $M := \Phi(E_1 \times E_2 \times E_3)$,
  the following assertions hold: 
  \begin{itemize}
  \item[\rm(a)] $\phi \: M \times \fk^{se} \to \g^{se},
    (m,x) \mapsto \Ad(m) x$ is a diffeomorphism.   
  \item[\rm(b)] For $K^{se} := K \cap G^{se}$, the map
    $\phi \: M \times K^{se} \to G^{se},   (m,k) \mapsto mkm^{-1}$ 
  is a diffeomorphism.   
  \end{itemize}
\end{prop}

\begin{prf} First we observe that \cite[Thm.~14.3.8]{HN12} provides
  $\Ad(K)$-invariant subspace $E_1, E_2$ and $E_3$,
  for which $\Phi$ is a $K$-equivariant   diffeomorphism. 
  
\nin  (a) {\bf Surjectivity:}
Let $x \in \g^{se}$. Then $x \in \comp(g) = \Ad(G) \fk
= \Ad(MK)\fk= \Ad(M)\fk$.
  Therefore $\phi$ is surjective.

 \nin {\bf Injectivity:} First we observe that, for
  $x \in \fk^{se}$ and $g \in G$, we have
  \begin{equation}
    \label{eq:ink}
    \Ad(g) x \in \fk \quad \Rarrow  \quad g \in K.
  \end{equation}
  In fact, for the element $k := \exp(x) \in K^{se}$
  the subalgebra $\Fix(\Ad(k))$ is compactly embedded by
  Theorem~\ref{thm:compemb}
  and $\Ad(g)x \in \fk$ implies $gkg^{-1} \in K$, so that
  Lemma~\ref{lem:compact-central} yields $g \in K$.

    If $\Ad(m_1) x_1 = \Ad(m_2) x_2$
  for $x_1, x_2 \in \fk^{se}$, then
  $\Ad(m_2^{-1}m_1)x_1 \in \fk$, so that the preceding
  argument yields $m_2^{-1} m_1 \in K$, and thus $m_1 = m_2$,
  which in turn implies $x_1 = x_2$.

  \nin {\bf Regularity:} By the Inverse Function Theorem, it
  remains to show that   the differential of $\phi$ is everywhere
  surjective. We have
  \[ \dd \phi(m,x)(m.v,y) = \Ad(m)(y + [v,x]),\]
  so that
  \[ \Im(\dd\phi(m,x)) = \Ad(m)(\fk + [m^{-1}.T(M),x]).\]
  The subspace $E := m^{-1}.T(M) \subeq \g = T_e(G)$ is a complement of $\fk$.
  We claim that $[x,E] + \fk = \g$.
  As $x$ is compact, there exists an $\ad x$-invariant complement
  $F \subeq \g$ of $\fk$. Then
  \[ \fz_\g(x) = \fz_F(x) \oplus \fz_\fk(x) \subeq \Fix(e^{\ad x}) \subeq \fk \]
  follows from $\Fix(e^{\ad x}) \subeq \fk$ (see Lemma~\ref{lem:cent-in-k}).
  As $F \cap \fk = \{0\}$, we obtain $\fz_F(x) = \{0\}$,
  so that $F = [x,F]$. This leads to
$[x,\g]  + \fk = F + \fk = \g.$ 
  We thus obtain for every complement $E$ of $\fk$ that
  \[ \g =  [x,\g] + \fk = [x,E] + [x,\fk] + \fk = [x,E] + \fk.\]
  This shows that $\dd\phi(m,x)$ is surjective.

  \nin (b) {\bf Surjectivity:}
  By Lemma~\ref{lem:c.1a}(iii), any
  $g \in G^{se}$ is conjugate to an element of $K^{se}$.
  Since $K^{se}$ is conjugation invariant in $K$, we find an $m \in M$
  and $k \in K^{se}$ with $\phi(m,k) = g$.

 \nin {\bf Injectivity:} For $k \in K^{se}$ Theorem~\ref{thm:compemb}
  implies that $\Fix(\Ad(k))$ is compactly embedded,
  so that Lemma~\ref{lem:compact-central} shows that
    \begin{equation}
    \label{eq:ink2}
    gkg^{-1} \in K \quad \Rarrow \quad g \in K.
  \end{equation}
    If $\phi(m_1,k_1) = \phi(m_2,k_2)$, then
  $(m_2^{-1}m_1)k_1 (m_2^{-1}m_1)^{-1} \in K$, so that the preceding
  argument yields $m_2^{-1} m_1 \in K$, and thus $m_1 = m_2$,
  which in turn implies $k_1 = k_2$.
  
  \nin {\bf Regularity:} By the Inverse Function Theorem, it
  remains to show that   the differential of $\phi$ is everywhere
  surjective. Using the group structure on $T(G)$, we have 
  \begin{align*}
 \dd \phi(m,k)(m.v,k.x)
&  = m(k.x)m^{-1} + m(vk - kv)m^{-1} 
  = m(k.x  + vk - kv)m^{-1} \\
    &  = m(k.(x  + \Ad(k)^{-1}v - v))m^{-1}.
  \end{align*}
  Therefore $\dd\phi(m,k)$ is surjective if and only if
  \[ \fk + (\Ad(k)^{-1} - \1)(m^{-1}.T(M)) = \g.\] 
  The subspace $E := m^{-1}.T(M) \subeq \g = T_e(G)$ is a complement of $\fk$.
  We claim that
  \[  (\Ad(k)^{-1} - \1)E + \fk = \g.\] 
  As $k \in K$, there exists an $\Ad(k)$-invariant complement
  $F \subeq \g$ of $\fk$. Then
  \[ \g^{\Ad(k)} = F^{\Ad(k)} \oplus \fk^{\Ad(k)} \]
  entails $F^{\Ad(k)}  = \{0\}$, so that $F = (\Ad(k)-\1)F$, which leads to
  \[  (\Ad(k)-\1)\g  + \fk = F + \fk = \g.\]
  We thus obtain for every complement $E$ of $\fk$ that
  \[ \g =  (\Ad(k)-\1)\g + \fk = (\Ad(k)-\1)E+
     (\Ad(k)-\1) \fk + \fk = (\Ad(k)-\1)E+ \fk.\]
  This shows that $\dd\phi(m,k)$ is surjective.
\end{prf}

\subsection{Another characterization of compact elements}

The theorem below provides an interesting characterization
of compact elements in $\g$ in terms of seemingly weaker conditions. 

\begin{thm}
  \mlabel{thm:compchar}
Let $G$ be a connected Lie group and
$x \in \g$. We consider the flow on $G$, defined by
\[ \alpha_t(g) := \exp(tx) g \exp(-tx)\quad \mbox{ for } \quad t \in \R, g \in G.\]
Then the following assertions hold:
\begin{itemize}
\item[\rm(a)] If $G$ is simply connected, then
$x \in \comp(\g)$ if and only if all orbits $\alpha_\R(g)$ have compact closure.
\item[\rm(b)] Let $\ad x = (\ad x)_s + (\ad x)_n$ be the Jordan decomposition
  of $\ad x$ in $\der(\g)$. Then all orbits of
  the flow   on $\Ad(G)\subeq \Aut(\g)$ defined by 
  \[ \oline\alpha_t(\Ad(g)) = e^{t \ad x}\Ad(g) e^{-t\ad x}
    = \Ad(\alpha_t(g)) \]
  has orbits that are relatively
  compact in $\Aut(G)$ if and only if
  \begin{equation}
    \label{eq:spec-nil}
 \Spec(\ad x) \subeq i \R \quad \mbox{ and } \quad
 (\ad x)_n(\g) \subeq \fz(\g).
  \end{equation}
\item[\rm(c)] If $\g = [\g,\g]$, then \eqref{eq:spec-nil} implies
  $x \in \comp(\g)$. 
\end{itemize}
\end{thm}

\begin{prf} Is it easy  to see that $x \in \comp(\g)$ implies that all
  orbits $\alpha_\R(g) \subeq G$ have compact closure. So the main
  point is the reverse argument.

\nin  {\bf Step 1:} First we show that relative compactness of the orbits in
  $\Ad(G)$ implies that all eigenvalues of $\ad x$ are purely imaginary: 
  Let $\g_\C^\lambda(x)$ be the generalized  $\lambda$-eigenspace
  of $\ad x$  in $\g_\C$. Then 
$\oline{\g_\C^\lambda(x)} =       \g_\C^{\oline\lambda}(x)$ 
  shows that, for $a > 0$, 
  \[ \fn := \g \cap \sum_{\Re \lambda \geq a} \g_\C^\lambda(x) \]
  is a real subalgebra of $\g$, consisting of nilpotent elements.

  If $a$ is the maximal real part of an eigenvalue and $a > 0$, then
  $\fn$ is abelian because
  $[\g_\C^\lambda, \g_\C^\mu] \subeq \g_\C^{\lambda+ \mu}$. Let $N := \exp \fn$ be the corresponding
  integral subgroup of $G$. As $\fn = [x,\fn]$,
    the subalgebra $\ad \fn$ consists of nilpotent
  derivations, so that $N := \exp(\ad \fn)$ is a unipotent group,
  hence in particular closed (\cite[Prop.~3.3.3, Prop.~11.2.15]{HN12})
  and the exponential function 
  $\exp \: \ad \fn \to \Ad(N)$ is a diffeomorphism.
  Moreover, $\fn \cap \fz(\g) = \{0\}$ implies that
  $\ad \res_{\fn} \: \fn \to \ad \fn$ is bijective.
  As $\lim_{t \to \infty} e^{t \ad x}y = \infty$ in the sense that the
  curve eventually leaves every compact subset, the same holds for
  its image in $\Ad(N)\subeq \Aut(\g)$. This implies that
  $\Re \lambda \leq 0$ for all eigenvalues of $\ad x$.
  The same argument applies to $-\ad x$ and thus entails
  that all eigenvalues of $\ad x$ are purely imaginary.

  \nin {\bf Step 2:} Assume now that $G$ is simply connected. 
We reduce the proof of (a) to the case where $\ad x$ is nilpotent: 
  From Step 1 we infer that, in the Jordan decomposition
  $D = D_s + D_n$ of $D := \ad x$ in $\der(\g)$
  (which is algebraic, hence splittable), the semisimple part
  $D_s$ is elliptic. As $G$ is simply connected,
  the derivation $D_s$ integrates to a flow $\beta \: \R \to \Aut(G)$.
  Since $D_s$ is elliptic, all orbits of this flow have compact closures.
  Moreover, $\alpha_\R$ and $\beta_\R$ commute, so that
  $\gamma_t := \beta_t^{-1} \alpha_t$ defines a flow on $G$
  for which all orbits are relatively compact.
  Its infinitesimal generator is the nilpotent derivation~$D_n$.

  \nin {\bf Step 3:} The nilpotent case: 
In view of \cite[Satz~III.1.1.4]{Kr84}, orbits of unipotent groups 
  are closed, which applies to the conjugation orbits
  \[ t \mapsto \Ad(\gamma_t(g)) = e^{t D_n} \Ad(g) e^{-t D_n}
  \in \Aut(G) \subeq \End(\g). \]
  As $\ad D_n$ is nilpotent on $\End(\g)$, all non-trivial orbits
  are non-compact and diffeomorphic to $\R$.
  Our assumption therefore implies that this action is trivial,
  which means that $e^{\R D_n}$ commutes with $\Ad(G)$. 
  We conclude that, for any $y \in \g$, we have
$0 = [D_n, \ad y] = \ad(D_n y),$ 
so that
\begin{equation}
  \label{eq:d-zent}
  D_n(\g) \subeq \fz(\g)  \quad \mbox{  and thus } \quad
  [\g,\g] \subeq \ker D_n.
\end{equation}
The fact that $D_n$ is the nilpotent component of
  $\ad x$, which annihilates $\fz(\g)$, shows that
  $\fz(\g) \subeq \ker D_n$, and thus $D_n^2 = 0$ by \eqref{eq:d-zent}.
  So
  \[ e^{tD_n} y = y + t D_n(y) \quad \mbox{ for } \quad t \in \R, y \in \g.\]
  For the automorphisms $\gamma_t  \in\Aut(G)$ with
  $\L(\gamma_t) = e^{tD_n}$ we then obtain 
  \[ \gamma_t(\exp y)  = \exp( y + tD_n(y))
    = \exp(y) \exp(t D_n(y)),\quad \mbox{ hence } \quad
    \gamma_\R(\exp y) = \exp y \exp(\R D_n(y)).\]
  If $D_n(y) \not=0$, then
  $\exp(\R D_n(y))$ is a non-trivial one-parameter group of
  $Z(G)_e \cong \R^d$ (\cite[Thm.~11.1.21]{HN12}), hence non-compact.
  This contradicts our assumption on relative compactness of the orbits.
  We thus obtain $D_n = 0$, i.e., that $x \in \comp(\g)$.

  \nin (b) Step 1 already implies that, if all orbits
  $\oline\alpha_\R(\Ad(g))$ are relatively compact in $\Aut(G)$,
  then $D_s$ is elliptic, i.e., $\Spec(\ad x) \subeq i \R$.
  Then  the flow $\oline\beta_t(\phi) = e^{t D_s} \phi e^{-tD_s}$
  also has compact orbit closures, and so has the linear unipotent flow 
  $\oline\gamma_t := \oline\beta_t^{-1} \oline\alpha_t$ generated by~$D_n$.
  The argument under Step 3 shows that this flow must be
  trivial, i.e., that
  $D_n(\g) \subeq \fz(\g)$.

  If, conversely, $x \in \g$ is such that $(\ad x)_s$ is elliptic and
  $(\ad x)_n(\g) \subeq \fz(\g)$, then $(\ad x)_n$ generates the trivial
  flow on $\Ad(G)$, and
  \[ \oline\alpha_t(\phi) = e^{t (\ad x)_s} \phi e^{-t (\ad x)_s}.\]
  But this flow has compact orbit closures because $(\ad x)_s$ is
  elliptic.

  \nin (c) Assume, in addition to \eqref{eq:spec-nil}, that
  $\g = [\g,\g]$. As the derivation $(\ad x)_n$ maps
  into the center, $[\g,\g] \subeq \ker (\ad x)_n$,
  so that $(\ad x)_n = 0$. In this case we obtain
  that $\ad x$ is elliptic, i.e., $x \in \comp(\g)$.
    \end{prf}

    \begin{ex} Here is a simple example that shows that, in general
      the condition of relative compactness of the orbits
      $\alpha_\R(g)$ in $G$  does not imply that $x \in \comp(\g)$.
  We consider the \break $3$-dimensional Heisenberg--Lie algebra
  \[ \g = \Bigg\{ \pmat{
    0 & a & c \\
    0 & 0 & b \\
    0 & 0 & 0} \: a,b,c \in \R \Bigg\} \]
  with the basis
  \[ p := \pmat{
    0 & 1 & 0 \\ 
    0 & 0 & 0 \\ 
    0 & 0 & 0}, \quad 
    q := \pmat{
    0 & 0 & 0 \\ 
    0 & 0 & 1 \\ 
    0 & 0 & 0} \quad \mbox{ and } \quad 
    z := \pmat{
    0 & 0 & 1 \\ 
    0 & 0 & 0 \\ 
    0 & 0 & 0}\]
    that satisfies
    \[ [p,q] = z, \quad [p,z] = [q,z] = 0.\]
    The corresponding Lie group
    \[ H := 
    \Bigg\{ \pmat{
    1 & a & c \\
    0 & 1 & b \\
    0 & 0 & 1} \: a,b,c \in \R \Bigg\} \]
    is simply connected and
    \[ \Gamma :=  \Bigg\{ \pmat{
    1 & 0 & c \\
    0 & 1 & 0 \\
    0 & 0 & 1} \: c \in \Z \Bigg\} \]
    is a discrete central subgroup, so that
    $G := H/\Gamma$ is a $3$-dimensional Lie group
    with $\pi_1(G) \cong \Gamma \cong \Z$. 

    The conjugation flow generated by $\ad p$ acts by
    \[ \alpha_t(\exp x) = \exp(e^{t \ad p}x)
    = \exp(x + t[p,x])= \exp(x) \exp(t[p,x]) \in \exp(x) Z(G).\]
    Since $Z(G) = \exp(\R z) \cong \R/\Z$ is compact, all orbits of this
    flows on $G$ are compact, but
    \[ p \not\in \comp(\g) = \R z.\]

    This examples also appears in \cite{Mo51} as a counterexample to a statement
    of A.~Weil asserting that if a locally compact group possesses a
    compact identity neighborhood invariant under all inner automorphisms,
    then there are arbitrarily small identity neighborhoods with this
    property.    \begin{footnote}{We thank A.~Valette for pointing out this
        reference.}     
    \end{footnote}
\end{ex}

\section{Connected components of $G^{se}$} 
\mlabel{sec:2}

In this section we address the description of the
connected components of the open subset~$G^{se}$
of stably elliptic elements.
In particular, we assume that $G^{se} \not=\eset$. 
As $G$ is connected, all connected components of $G^{se}$ are
invariant under conjugation. Using the fact that
$\g$ has a compactly embedded Cartan subalgebra $\ft$
(Proposition~\ref{prop:c.1ax}) and
all elements of $\g^{se}$ are conjugate to elements in $\ft$,
we use the root decomposition with respect to $\ft$ to 
label connected components of  $G^{se}$ by $\Z$-valued functions
on the set of non-compact roots (Theorem~\ref{thm:conncomp}).

\subsection{Root decomposition}

\begin{defn} (a) 
Let $\ft \subeq \g$ be a compactly embedded Cartan 
subalgebra, $\g_{\C}$ the complexification of $\g$, 
$z= x + i y \mapsto z^* := - x + iy$
the corresponding involution, and $\ft_\C$ the 
corresponding Cartan subalgebra of $\g_{\C}$. 
For a linear functional $\alpha \in \ft_\C^*$, we define the
{\it root space} 
\[ \g_{\C}^\alpha := \{ x \in \g_{\C} : (\forall y \in \ft_\C)\ [y,x] = 
\alpha(y)x\} \] 
 and write  
\[ \Delta := \Delta(\g_\C, \ft_\C) 
:= \{ \alpha \in \ft_\C^*\setminus \{0\}: 
\g_\C^\alpha \not= \{0\}\} \] 
for the set of {\it roots of $\g$.} 

\nin (b)  A root $\alpha \in \Delta$ is called {\it semisimple} 
if $\alpha([z,z^*]) \not= 0$ holds for an element $z \in \g_\C^\alpha$. 
In this case, $[\g_\C^\alpha, \g_\C^{-\alpha}] = \C [z, z^*]$ contains a 
unique element $\alpha^\vee$
with $\alpha(\alpha^\vee) = 2$ which we call the 
{\it coroot of $\alpha$}. 
We write $\Delta_s$ for the set of semisimple roots and call the roots in 
$\Delta_r := \Delta \setminus \Delta_s$ the {\it solvable roots}.  

A semisimple root $\alpha$ is called {\it compact} 
if $\alpha^\vee \in \R^+ [z,z^*]$, i.e.\  if 
$\alpha([z, z^*]) > 0$. All other roots are called {\it non-compact}. 
We write $\Delta_k$, resp., $\Delta_p$ for the set of compact, resp., 
non-compact roots. We also set $\Delta_{p,s} := \Delta_p \cap
\Delta_s$.   

\nin (c) For each compact root $\alpha \in \Delta_k$, the linear mapping 
$ s_\alpha \colon \ft \to \ft, x \mapsto x - \alpha(x) \alpha^\vee$ 
is a reflection in the hyperplane $\ker \alpha$. 
We write $\cW_\fk$ for the group generated by these
reflections. It is called the {\it Weyl group} 
of the pair $(\fk,\ft)$. 
According to \cite[Prop.~VII.2.10]{Ne99},
this group is finite. 
\end{defn}

\begin{defn} \mlabel{def:roots}
  (a)  A subset $\Delta^+ \subeq \Delta$ is 
called a {\it positive system} 
if there exists an element $x_0 \in i\ft$ with 
\[  \Delta^+ = \{ \alpha \in \Delta : \alpha(x_0) > 0\} \] 
and $\alpha(x_0) \not= 0$ holds for all $\alpha \in \Delta$. 
A positive system $\Delta^+$ is said to be 
{\it adapted} 
if for $\alpha\in \Delta_k$ 
and $\beta \in \Delta_p^+$ we have $ \beta(x_0) > \alpha(x_0)$ for 
some $x_0$ defining $\Delta^+$. In this case, we call 
$\Delta_p^+ := \Delta^+\cap \Delta_p$ an {\it adapted system of
  positive non-compact roots}.
In the following $\Delta_p^+$ always denotes an adapted positive system;
such systems always exist if $\fg$ is quasihermitian
(\cite[Prop.~VII.2.14]{Ne99}). 

\nin (b) We associate to an adapted system $\Delta_p^+$ 
of positive non-compact roots the convex cones 
\[  W^{\rm min}_\ft := W^{\rm min}_\ft(\Delta_p^+) := 
\cone(\{ i[{z_\alpha}, z_\alpha^*] \colon z_\alpha \in 
\g_\C^\alpha, \alpha\in \Delta_p^+ \}) \subeq \ft, \] 
and 
\[  W^{\rm max}_\ft := W^{\rm max}_\ft(\Delta_p^+) := \{ x \in  \ft \colon (\forall \alpha \in 
\Delta_p^+)\, i \alpha(x) \geq 0\}. \] 
\end{defn}

For $x \in \ft$ we then have
\[ \ker(\ad x) = \ft \oplus \g \cap \Big(\sum_{\alpha(x) = 0} \g_\C^\alpha\Big).\]
We conclude that the
cone $(W_\ft^{\rm max})^\circ \subeq \ft$ is a connected component of
$\ft \cap \comp(\g)^\circ$.

\subsection{Specification of connected components}

  An element $x \in \ft$ is contained in the interior of $\comp(\g)$
  if and only if its centralizer is compactly embedded
(\cite[Lemma~VII.1.7]{Ne99}), which means that
  $\alpha(x) \not=0$ for all non-compact roots $\alpha \in \Delta_p$.
  For the analogous picture on the group level, we observe that,
  for $x \in \ft$ and $g = \exp x$, we have
  \begin{equation}
    \label{eq:fixadg} \g^{\Ad(g)} = \ft \oplus
    \g \cap \Big(\sum_{\alpha(x) \in 2\pi i \Z}\g_\C^\alpha\Big).
  \end{equation}
The connected components of the set
\begin{equation}
  \label{eq:tse}
 \ft^{se} := \ft \cap \g^{se} = \ft \cap \exp^{-1}(G^{se})
 = \ft \setminus \bigcup_{\alpha \in \Delta_p} \alpha^{-1}(2\pi i \Z),
\end{equation}
are alcoves specified by inequalities of the form
\[ \ft^{se}(\vartheta) := \{ x \in \ft \:
(\forall \alpha \in \Delta_p^+)\
i\alpha(x) \in (2\pi \vartheta_\alpha, 2\pi(\vartheta_\alpha+1))\},\]
where $\vartheta = (\vartheta_\alpha) \in \Z^{\Delta_p^+}$. 
We call 
\begin{equation}
  \label{eq:basic} 
 \ft^{se}(0) = \{ x \in \ft \: (\forall \alpha \in \Delta_p^+)\
 0 < i\alpha(x) < 2\pi \}
\end{equation}
the {\it basic alcove}. Note that all the alcoves $\ft^{se}(\vartheta)$ are
open polyhedra because $\Delta_p$ is a finite set. We also observe that
\begin{equation}
  \label{eq:invbas-alc}
-\ft^{se}(0) = \ft^{se}(-1).  
\end{equation}
For every connected component of
$G^{se}$, the inverse image in $\ft^{se}$
is a union of certain alcoves $\ft^{se}(\vartheta)$.
So one has to determine which alcoves map into the same connected
component of~$G^{se}$.

Let $T := \exp \ft$ and $\Gamma_T := \exp^{-1}(e) \cap \ft$ be the
kernel of the exponential function of~$T$. This is a discrete
subgroup of $\ft$, contained in
\[ \Gamma_Z := \exp^{-1}(Z(G))
= \{ x \in \ft \: (\forall \alpha \in \Delta)
\ \alpha(x) \in 2 \pi i \Z\}.\]
Note that
\[ \ft^{se} = \ft^{se} + \Gamma_Z = \ft^{se} + \Gamma_T\]
and that $\Gamma_T = \Gamma_Z$ if and only if $Z(G) = \{e\}$. 
For $z \in \Gamma_Z$ we have
\begin{equation}
  \label{eq:alc-trans}
 z + \ft^{se}(\vartheta) = \ft^{se}(\vartheta')
  \quad \mbox{ with } \quad
  \vartheta_\alpha' = 
  \vartheta_\alpha + i \frac{\alpha(z)}{2\pi}.
\end{equation}

  \begin{lem} \mlabel{lem:tse-comp}
    The subsets
    \[ T^{se}(\vartheta) := \exp(\ft^{se}(\vartheta)) \]
    are the connected components of the open subset 
    $T^{se} := T \cap G^{se}$ of $T$.
    Moreover, the following are equivalent:
    \begin{itemize}
    \item[\rm(a)] $T^{se}(\vartheta) = T^{se}(\vartheta')$. 
    \item[\rm(b)] $\ft^{se}(\vartheta') \subeq \ft^{se}(\vartheta) + \Gamma_T$. 
    \item[\rm(c)] $\ft^{se}(\vartheta') \cap  (\ft^{se}(\vartheta) + \Gamma_T)
      \not=\eset$. 
    \end{itemize}
  \end{lem}

  \begin{prf} As $\exp_T^{-1}(T^{se}(\vartheta)) = \ft^{se}(\vartheta) + \Gamma_T$,
    the implications (a) $\Rarrow$ (b) $\Rarrow$ (c) are trivial.
    If (c) holds, then there exists a $z \in \Gamma_T$ and
    $x \in \ft^{se}(\vartheta)$ with $x + z \in \ft^{se}(\vartheta')$.
    Then $x$ is contained in the connected component $\ft^{se}(\vartheta') - z$
    of $\ft^{se}$, and thus
$\ft^{se}(\vartheta) = \ft^{se}(\vartheta') - z.$ 
    Applying $\exp_T$ to both sides now yields~(a).
  \end{prf}

\begin{prop} \mlabel{prop:3.4}
The subsets 
\[ \g^{se}(\vartheta) := \Ad(G)\ft^{se}(\vartheta) \subeq \g^{se}\]
have the following properties:
\begin{itemize}
\item[\rm(a)] They are the connected components of the open subset~$\g^{se}$,
  hence in particular open.
\item[\rm(b)] $\g^{se}(\vartheta_1) = \g^{se}(\vartheta_2)$
  if and only if there exists $\gamma \in \cW_\fk$ with $\gamma.\ft^{se}(\vartheta_1) = \ft^{se}(\vartheta_2)$. 
\end{itemize}
\end{prop}

\begin{prf} (a) Clearly, $\g^{se}(\vartheta)$
is connected. To see that it  is open, we write is as
\[ \g^{se}(\vartheta) = \Ad(G)\fk^{se}(\vartheta) \quad \mbox{ with }  \quad
\fk^{se}(\vartheta) = \Ad(K)\ft^{se}(\vartheta).\]
In view of Proposition~\ref{prop:2.11}(b), it suffices to see that
$\fk^{se}(\vartheta)$ is open, which follows from \cite[Lemma~III.5]{Ne96}
(see also Remark~\ref{rem:3.2} below).
We conclude that the sets $\g^{se}(\vartheta)$ are the connected components of $\g^{se}$.

\nin (b) As their are connected components of $\g^{se}$,
two sets $\g^{se}(\vartheta_1)$ and $\g^{se}(\vartheta_2)$ 
coincide if they intersect. 
For $x \in \ft$ we have $\Ad(G)x \cap \ft = \cW_\fk.x$
(\cite[Lemma~VIII.1.1]{Ne99}), so that
$\g^{se}(\vartheta_1) = \g^{se}(\vartheta_2)$ if and only if
\[ \ft^{se}(\vartheta_1) \cap \cW_\fk.\ft^{se}(\vartheta_2) \not=\eset.\] 
As the  action of $\cW_\fk$ on $\Delta_p$
induces an action on the set of connected components of $\ft^{se}$,
the assertion follows.   
\end{prf}

\begin{rem}  \mlabel{rem:3.2}
The derivative of the exponential function
$\exp \:  \g \to G$ is given by
\[  \dd\exp(x)y = \exp x.\frac{\1-e^{-\ad x}}{\ad x}y,\]
hence invertible if and only if
\[
x \in \regexp(\g) := \bigl\{ x \in\g\: \Spec(\ad x) \cap2\pi i \Z
\subeq\{0\}\bigr\}.\]
From \eqref{eq:cent2} we thus obtain that, for compact elements
$x \in  \g$:
\[ x \in \regexp(\g) \quad \Leftrightarrow \quad
  \g^{e^{\ad x}} = \ker(\ad x).\]
For $x \in \ft^{se}$, the left hand side $\g^{e^{\ad x}}$ is compactly embedded,
but may be larger than $\ker(\ad x)$ if compact roots vanish on~$x$.
\end{rem}

  \begin{lem} \mlabel{lem:expregx} 
      If $\Delta^+$ is an adapted positive system
    and $\Delta_k \subeq \Delta_p^+ - \Delta_p^+,$ 
    then
    \[    \fg^{se}(0) \subeq \regexp(\g).\]
\end{lem}

\begin{prf} Let $x \in \g^{se}(0) = \Ad(G) \ft^{se}(0)$.
  To see that $\exp$ is regular in $x$, we may w.l.o.g.\ assume that
  $x \in \ft^{se}(0)$. Clearly,
  $i\alpha(x) \in (0,2\pi)$ for $\alpha \in \Delta_p^+$,
  so that $\{\alpha(x)\} \cap 2\pi i \Z \subeq \{0\}$.
  For $\alpha \in \Delta_k$, our assumption yields
  $\beta,\gamma \in \Delta_p^+$ with $\alpha = \beta - \gamma$.
  As both $i\beta(x)$ and $i\gamma(x)$ are contained in the
  interval $(0,2\pi)$, their difference
  $i\alpha(x)$ is contained in the open interval $(-2\pi,2\pi)$,
  and therefore $\{\alpha(x) \} \cap 2\pi i \Z \subeq \{0\}$
  also holds in this case. This implies that
  $\Spec(\ad x) \cap 2\pi i \Z \subeq \{0\}$, i.e.,
  $x \in \regexp(\g)$.   
\end{prf}

\begin{rem} \mlabel{rem:cont-quot}
  Let $T_+ \subeq T$ be an alcove that meets every conjugacy class of $K$
  exactly once, i.e., $T_+ =  \exp(\ft_{++})$, where
  $\ft_{++} \subeq \ft$ is a closed convex subset that is a fundamental domain
  for the action of the affine Weyl group $\Gamma_T \rtimes \cW_\fk$
  (\cite[Lemma~12.4.25]{HN12}). As $K = V \times C$, where $V \cong \R^d$ and
  $C$ is compact, we have $T_+ = V \times T_+^C$, where $T_+^C$ is an alcove
  of the compact group $C$. Let
  \[ q \: K \to T_+ \cong K/{\rm conj} \]
  be the quotient map, mapping $k$ to $uku^{-1} \in T_+$.
  We claim that $q$ is a continuous map. In view of the product decomposition
  $K = V \times C$ and $q(v,c) = (v,q(c))$, we only have to conjugate
  with elements of $C$. Assume that $k_n = (v_n,c_n) \to k = (v,c)$,
  i.e., $v_n \to v$ and $c_n \to c$. Then
  \[ q(k_n) = (v_n, t_n) = (v_n, d_n c_n d_n^{-1})
    \quad \mbox{ with } \quad d_n \in C.\]
  We have to show that $q(k_n) \to q(k)$. Passing to a subsequence,
  we may assume that $d_n \to d$ in the compact group~$C$.
  Then $d_n c_n d_n^{-1} \to d c d^{-1}$ implies $dcd^{-1} \in T_+$
  (a closed subset of $T$), hence
  $q(k_n) \to (v, dcd^{-1}) = q(k)$.   
\end{rem}

\subsection{Classification of connected components}

As the exponential function need not be regular on the open subsets
$\g^{se}(\vartheta)$, for instance if $G$ is a compact non-abelian group,
it is not clear if their exponential image
is open, but the following theorem shows in particular that this is  always
the case. 

\begin{thm} \mlabel{thm:conncomp}
  The  connected components of $G^{se}$ are of the form 
  \begin{equation}
    \label{eq:gs-comp}
 G^{se}(\vartheta):= \exp(\g^{se}(\vartheta))  
 = \exp(\Ad(G)\ft^{se}(\vartheta)) = \bigcup_{g \in G} g T^{se}(\vartheta) g^{-1}
  \end{equation}
for some tuple $(\vartheta_\alpha)_{\alpha \in \Delta_p}$ of integers.
If $\cW_\fk \cong N_K(T)/T$ is the Weyl group of $(\g,\ft)$, acting on $T$,
then
\[ G^{se}(\vartheta_1) = G^{se}(\vartheta_2)\quad \Leftrightarrow \quad
  T^{se}(\vartheta_1) \in \cW_\fk.T^{se}(\vartheta_2).  \] 
\end{thm}

\begin{prf} Proposition~\ref{prop:2.11}(c) implies that intersecting with $K$
  yields a bijection between connected components of $G^{se}$ and those of $K^{se}$,
  hence   reduces the problem to describe
the connected components of $G^{se}$ to the connected components
of $K^{se}$.

Since every conjugacy class in $K$ intersects $T$, we have
\[ K^{se} = \bigcup_{k \in K} k T^{se} k^{-1} \quad \mbox{ for }\quad
  T^{se} = T \cap G^{se}.\]
By Lemma~\ref{lem:tse-comp}, the connected components of $T^{se}$ are
of the form
\[ T^{se}(\vartheta) := \exp(\ft^{se}(\vartheta)) \quad \mbox{ with }\quad
T^{se}(\vartheta_1) = T^{se}(\vartheta_2) \quad \Leftrightarrow \quad
(\ft^{se}(\vartheta_1) + \Gamma_T) \cap \ft^{se}(\vartheta_2)\not=\eset.\] 

For $x \in \ft$ we have $\exp x \in T^{se}$ if and only if $x \in \ft^{se}$.
In this context we may also fix an ``alcove''
$T_+ \subeq T$ that intersects every $K$-conjugacy class exactly
once and obtain a continuous map $K \to T_+$, mapping
each element to the unique representative of its conjugacy class
(Remark~\ref{rem:cont-quot}).
From the compactness of $\Ad(K) \cong K/Z(K)$, it now follows that
the connected subsets
\[ K^{se}(\vartheta) := \exp(\Ad(K) \ft^{se}(\vartheta)) \subeq K \]
are open in $K$,
hence connected components of $K^{se}$. This proves the first part of the theorem.

To characterize when $G^{se}(\vartheta_1) = G^{se}(\vartheta_2)$,
we first observe that this is equivalent to
$K^{se}(\vartheta_1) = K^{se}(\vartheta_2)$
(Proposition~\ref{prop:2.11}(c)).
For $k \in T$ and $\Conj(k) := \{ uku^{-1} \: u \in K\}$, we have
\[ \Conj(k) \cap T = \cW_\fk.k \]
(cf.~\cite[Prop.~12.2.13]{HN12}).\begin{footnote}
{If $uku^{-1} \in T$, then
$u^{-1} T u \subeq Z_K(k)$ is a Cartan subgroup, hence conjugate to $T$,
so that there exists $v \in Z_K(k)$ with
$v^{-1} u^{-1} T uv = T$, and then $uv \in N_K(T)$ satisfies
$uvk(uv)^{-1} = uku^{-1}$.}
\end{footnote}
Therefore
\[ K^{se}(\vartheta) \cap T = \cW_\fk.T^{se}(\vartheta),\]
and thus 
\[ K^{se}(\vartheta_1) = K^{se}(\vartheta_2)\quad \Leftrightarrow \quad
  T^{se}(\vartheta_1) \in \cW_\fk.T^{se}(\vartheta_2).  \] 
This completes the proof.
\end{prf}

  \begin{rem} Recall from \eqref{eq:alc-trans} that, for
    $z \in \Gamma_Z$, we have 
\[  z + \ft^{se}(\vartheta) = \ft^{se}(\vartheta')
  \quad \mbox{ with } \quad
  \vartheta_\alpha' = 
  \vartheta_\alpha + i \frac{\alpha(z)}{2\pi}.\]
On the group level, this implies
\[ \exp(z) T^{se}(\vartheta) = T^{se}(\vartheta'),\] 
and Theorem~\ref{thm:conncomp} now implies that
\begin{equation}
  \label{eq:gtheta-trans}
  \exp(z) G^{se}(\vartheta) = G^{se}(\vartheta').
\end{equation}
\end{rem}

\section{The case of simple hermitian Lie algebras}
\mlabel{sec:3}

In this section we specialize to the case where $\g$ is a simple 
real Lie algebra which is {\it hermitian}, i.e.,
the maximal compactly embedded subalgebras $\fk$ have a non-zero
center $\fz(\fk)$, and then $\fz(\fk)$ is one-dimensional
(\cite{Hel78}, \cite{Ne99}). 
We fix a connected Lie group $G$ with Lie algebra $\g$, and choose
$z \in \fz(\fk)$ such that $\Spec(\ad z) = \{0, \pm i \}$.
We pick a positive system $\Delta^+$ in such a way that
\[ \Delta_p^+ = \{ \alpha \in \Delta_p \: i\alpha(z) = 1\}.\]
Here we are mostly interested in the connected component
$G^{se}(0)$. Proposition~\ref{prop:3.11} shows that
this domain ``does not depend on $G$'', in the sense that
$\exp \: \g^{se}(0) \to G^{se}(0)$ is always a diffeomorphism. 
The Properness Theorem~\ref{properness:theorem},
that will be used below, provides important information
on the polar decomposition of $G^{se}(0)$ by asserting that
compact subsets of $K^{se}(0)$ correspond to closed subsets of~$G$.

\begin{rem}
  In general, different connected components of $G^{se}$
  may not be obtained from each other by multiplication with central
  elements. This follows from the fact that $\Ad(G)^{se}$ may not 
  be  connected, but the center of $\Ad(G)$ is trivial if $\g$ is a 
simple hermitian Lie algebra.
\end{rem}

\begin{ex} \mlabel{ex:herm}
  \nin (a) The Lie algebra
  \[ \g = \sp_{2n}(\R) = \{ x \in \gl_{2n}(\R) \: x^\top J + J x = 0 \}
    \quad \mbox{ for }  \quad J = \pmat{0& -1_n \\ 1_n & 0} \] 
  can also be characterized as the subspace $J \Sym_{2n}(\R)$ of
  $M_{2n}(\R)$. For the complex structure $J$ on $\R^{2n}$, 
  the subalgebra $\fk = \fu_n(\C) = \sp_{2n}(\R) \cap \gl_n(\C)$
  is maximal compactly embedded and
  the space $\ft$ of diagonal matrices
  $\bx = \diag(-i x_1, \ldots, -i x_n)$ is a compactly embedded
  Cartan subalgebra. 
  Then $\ft \cong i \R^n$ in a natural way, and the positive system
  $\Delta_p^+$ can be chosen such that
  \[ \Delta_p^+ = \{  \eps_i + \eps_j, \:
    1 \leq i \leq j \leq n\},\quad \mbox{   where } \quad
    \eps_j(\bx)  =  -i x_j.\]
Note that 
\[ \Delta_k = \{ \pm(\eps_j - \eps_k) \: j < k\} \cong A_{n-1}\]
is the root system of $\fu_n(\C)$,
and the Weyl group $\cW_\fk \cong S_n$ acts by permutation of the diagonal
entries. Here
  \begin{equation}
    \label{eq:ts0}
    \ft^{se}(0)= \{ \bx \: (\forall j)\ 0 < x_j < \pi\}.
  \end{equation}
    The element $z = \frac{1}{2}\diag(-i,\ldots, -i)$ satisfies
  $i\alpha(z) = 1$ for all $\alpha \in \Delta_p^+$.
  In terms of the complex structure $J$,   we have
  \[ W_\ft^{\max}=W_\ft^{\min} = \{ x \in \ft \: Jx \geq 0\},\]
  and this implies that
  Hence $W^{\rm max} = W^{\rm  min} = \{ x \in \sp_{2n}(\R) \: J x \geq 0\}$. 
Therefore $\mathfrak{g}^{se}(0)=\mathrm{Ad}(G)\mathfrak{t}^{se}(0)$
is contained in $W^{\rm max}$ and we derive that any 
\[ g\in G^{se}(0) = \exp(\g^{se}(0))\]
has the property that $\Spec(g) \subeq  \T \setminus \{-\1\}.$ 
This in turn implies that
$G^{se}(0)$ coincides with the domain
$\mathrm{Sp}_{\mathrm{ell}}^+(2n)$ in \cite{ABP22, He22}.

  We take a closer look at the case $n = 2$, where
  \[ \Delta_p^+ = \{ \eps_1 + \eps_2, 2 \eps_1, 2 \eps_2 \} \]
  is not linearly independent.
  Here
  \[\frac{1}{2\pi i} \Gamma_Z =  \{ (x_1, x_2) \:
    x_1, x_2 \in \shalf\Z, x_1 + x_2 \in \Z\},\]
  and
  \begin{align*}
 \frac{1}{2\pi i} \ft^{se}(0,0)
&    = \{ (x_1, x_2) \: 0 < x_1, x_2 < 1/2, 0 < x_1 + x_2 < 1\} \\
&    = \{ (x_1, x_2) \: 0 < x_1, x_2 < 1/2\} \cong \Big(0,\frac{1}{2}\Big)^2
      \end{align*}
  is a square in our coordinates.
  For
  \[ \alpha_1 := 2 \eps_1, \quad \alpha_2 := 2 \eps_2, \quad
    \alpha_3 := \eps_1 + \eps_2,\]
  we obtain 
  \[ \frac{1}{2\pi i} \ft^{se}(1,0,0)
    = \big\{ (x_1, x_2) \: \shalf < x_1 < 1, 0 < x_2 < \shalf,
    0 < x_1 + x_2 < 1\big\}\]
  which is a triangle. It follows in particular that
  $\ft^{se}(1,0,0)$ is not a translate of $\ft^{se}(0,0,0)$.

\nin (b) Let $\g = \su_{p,q}(\C)$, $n = p + q$, where 
$\ft \subeq \g$ is the space is diagonal matrices.
We identify $\ft$ with
\[ \{ x \in i \R^n \: x_1 + \cdots + x_n = 0\}.\] 
The positive system
  $\Delta_p^+$ can be chosen such that
  \[ \Delta_p^+ = \{ \eps_i - \eps_j \: i \leq p < j\},\quad \mbox{ 
    where } \quad \eps_j(\diag(ix_1, \ldots, i x_n))  =  i x_j.\]     

  We take a closer look at the case $p = 2$, $q = 1$, $n = 3$, where 
  \[ \Delta_p^+ = \{ \eps_1 - \eps_3, \eps_2 - \eps_3 \}, \]
  and
  \[     \alpha_1 := \eps_1 - \eps_3, \quad \alpha_2 := \eps_2 - \eps_3 \]
are linearly idenpendent. It follows in particular that
  all $\ft^{se}(\vartheta)$ are non-empty.

  To get a geometric picture, we eliminate $x_3 = - x_1 - x_2$
  and identify $\ft$ with $i \R^2$. Then
  \[ \alpha_1 = 2 \eps_1 + \eps_2 \quad \mbox{ and } \quad 
    \alpha_2 =  \eps_1 + 2\eps_2,\]
  and the subsets $\ft^{se}(\vartheta)$ become parallelograms
  whose vertices are contained in $\Gamma_Z$. We therefore have
  \[ \ft^{se}(\vartheta) = z + \ft^{se}(0) \quad \mbox{ for }\quad
    \vartheta_j = \frac{i \alpha_j(z)}{2\pi}.\] 
\end{ex}

\begin{ex} (a) For $G := \tilde\SL_2(\R)$, the
  compactly embedded Cartan subalgebra $\ft = \so_2(\R) = \fk$
  is $1$-dimensional and
  $\Delta = \{ \pm \alpha\}$.
  For $z \in \ft$ with $i\alpha(z) = 1$, we then have
  \[ \ft^{se}(\vartheta) = (2\pi \vartheta, 2\pi(\vartheta+1)) \cdot z.\]
  As $\Delta_k = \eset$, the Weyl group $\cW_\fk$ is trivial,
  so that all connected components
  $\g^{se}(\vartheta)$, $\vartheta \in \Z$, are distinct.
  Here the exponential function $\exp_T \: \ft\to T$ is a diffeomorphism,
  so that also the connected components
  $G^{se}(\vartheta)$, $k \in \Z$, are distinct.
  The center $Z(G) \cong \Z$, which is generated by
  $\exp(2\pi z)$, acts simply transitively on the connected components
  of $G^{se}$:
  \[ G^{se}(\vartheta) = \exp(2 \pi z)^\vartheta  G^{se}(0), \quad \vartheta \in \Z.\] 

  \nin (b) The picture changes for $G = \SL_2(\R)$, where
  $\ker(\exp_T) =  \Z 4\pi z$. Then
  \[ T^{se} = T^{se}(0) \cup T^{se}(1) \quad \mbox{ leads to } \quad 
    G^{se} = G^{se}(0) \cup G^{se}(1).\]

  \nin (c) For $G = \PSL_2(\R)$ with 
  $\ker(\exp_T) =  \Z 2\pi z$, we obtain 
  \[ T^{se} = T^{se}(0) \quad \mbox{ leads to } \quad 
    G^{se} = G^{se}(0),\]
  so that $G^{se}$ is connected. 
\end{ex}

  \begin{lem} \mlabel{lem:l1}
    Suppose that $\g$ is simple hermitian. 
    If $x \in \ft^{se}(0)$ and $z \in \Gamma_Z$ satisfy
    $\exp(x + z) \in T^{se}(0)$, then $z \in \Gamma_T$, so that
    $\exp(x+z) = \exp(x)$. The same conclusion holds if
    \begin{equation}
      \label{eq:tse-shift}
 T^{se}(0) = T^{se}(\vartheta) \quad \mbox{ for } \quad
 \vartheta_\alpha = i \frac{\alpha(z)}{2\pi}.
    \end{equation}
\end{lem}

\begin{prf} From \eqref{eq:alc-trans}, we obtain
\[  z + \ft^{se}(0) = \ft^{se}(\vartheta)
  \quad \mbox{ with } \quad
  \vartheta_\alpha = i \frac{\alpha(z)}{2\pi},\]
and our assumption implies that 
$\exp(x + z) \in T^{se}(\vartheta) \cap T^{se}(0),$ 
so that the two connected components
$T^{se}(0)$ and $T^{se}(\vartheta)$ of $T^{se}$ coincide
(Lemma~\ref{lem:tse-comp}). 
With Lemma~\ref{lem:tse-comp}
this in turn implies that
\[ \ft^{se}(\vartheta) = \ft^{se}(0) + z' \quad \mbox{ for some}
  \quad z' \in \Gamma_T.\]
So $\alpha(z) - \alpha(z')$ for all $\alpha \in \Delta_p^+$. 
Next we use that $\Delta_p^+$ spans $i\ft^*$ (because
$\g$ is simple hermitian) to conclude that $z  =z' \in \Gamma_T$.   
\end{prf}

\begin{rem} Let $\ft_0 := (\Delta_p)^\bot \subeq \ft$.
  Then $\ft_0$ commutes with the ideal of $\g$ generated by
  any $\fk$-invariant complement $\fp$ of $\fk$.
  Let  $\g = \fr \rtimes \fs$ be a $\fk$-invariant  
  Levi decomposition (\cite[Lemma~14.3.3]{HN12}).
Write $\fs_c \subeq \fs$ for the maximal compact semisimple
  ideal and $\fs_{nc} \trile \fs$ for an ideal complementing $\fs_c$.
  Then $\ft_0$ commutes with $\fr \subeq \fp + \ft$
  and also with $\fs_{nc} \subeq \fp + [\fp,\fp]$. This shows that
  $\g = (\fr \rtimes \fs_{nc})  \rtimes \fs_c$, where $\fs_c$ is compact semisimple
  and $\ft_0$ commutes with $\fr + \fs_{nc}$. 
\end{rem}

\begin{prop} \mlabel{prop:3.11}
  Suppose that $\g$ is simple hermitian.
  The the following assertions hold:
  \begin{itemize}
  \item[\rm(a)] $\exp \: \g^{se}(0) \to G^{se}(0)$ is a diffeomorphism.
  \item[\rm(b)] The adjoint representation
    $\Ad \: G \to \Ad(G) \cong G/Z(G)$ maps $G^{se}(0)$ diffeomorphically
    onto $\Ad(G)^{se}(0)$.
  \end{itemize}
\end{prop}

\begin{prf} (a) As $\g$ is simple hermitian, we have
  \[ \g_\C = \fk_\C \oplus \fp_\C^+  \oplus \fp_\C^- \quad \mbox{ with } \quad
    \fp_\C^\pm = \sum_{\alpha \in \pm \Delta_p^+} \g_\C^\alpha.\]
  Then $\fp_\C^+ + [\fp_\C^+, \fp_\C^-] + \fp_\C^- \trile \g_\C$ 
  is a non-zero ideal because this subspace is invariant
  under $\fk_\C$ and $\fp_\C^\pm$. The simplicity of $\g_\C$
  now implies that this ideal cannot be proper, i.e., 
$\fk_\C = [\fp_\C^+, \fp_\C^-],$ 
  and hence that $\Delta_k \subeq \Delta_p^+ - \Delta_p^+$. 
  Now Lemma~\ref{lem:expregx}  implies that the exponential
  function has no singular point in $\g^{se}(0)$.
  So the map $\exp \: \g^{se}(0) \to G^{se}(0)$ is a surjective
  local diffeomorphism, and it remains to show that it is injective.

  Let $x,y \in \g^{se}(0)$ with $\exp x = \exp y$.
  Since $\exp$ is regular in $x$,
  \cite[Lemma~9.2.31]{HN12} implies that $[x,y] = 0$,
  so that $\exp(x-y) = e$. Therefore
  $\fb := \R x + \R y$ is a compactly embedded abelian subalgebra,
  hence conjugate by $\Ad(G)$ to a subspace of $\ft$
  (\cite[Thm.~VII.1.4]{Ne99}).
  We may therefore assume that $x,y \in \ft^{se}(0)$.
  Now $\exp(x-y) = e$ implies that
  $\alpha(x-y) \in 2 \pi i \Z$ for all roots~$\alpha$.
  If $\alpha \in \Delta_p^+$, then
  $i\alpha(x-y) \in (-2\pi, 2\pi) \cap 2\pi \Z = \{0\}$.
  As $\Delta_k \subeq \Delta_p^+ - \Delta_p^+$,
 all roots vanish on $x-y$, so that
  $x-y \in \fz(\g) = \{0\}$. This proves that
  $\exp$ is injective on $\g^{se}(0)$.

\nin (b) As $\Ad \circ \exp_G = \exp_{\Ad(G)}$ and (a) applies to
$G$ as well as the adjoint group, the assertion follows from the fact
that $\Ad \: G \to \Ad(G)$ is a covering map.
\end{prf}

\begin{thm}\mlabel{properness:theorem}
  {\rm(Properness Theorem)}
  Suppose that $\g$ is simple hermitian.
  If $C \subeq K^{se}(0)$ is a compact subset,
    then $\{ p C p^{-1} \: p  \in \exp(\fp)\}  \subeq
    G^{se}(0)$ is a closed subset of~$G$.
\end{thm}

\begin{prf} Suppose that $g_n = p_n c_n p_n^{-1}$ with
$c_n \in C\subeq K^{se}(0)$ and $p_n = \exp y_n$ with $y_n \in \fp$
is a sequence
  converging to some $g \in G$. Passing to a subsequence, we may
  assume that $c_n \to c \in C$. We claim that the sequence
  $y_n$ has a convergent subsequence, converging to some
  $y \in \fp$, so that $g = \exp(y) c \exp(-y) \in G^{se}(0)$.

In view of (b), we may w.l.o.g.\ assume that $Z(G) = \{e\}$, i.e.,
  that $G \cong \Ad(G)$.   
  We write $p_n = u_n a_n u_n^{-1}$ with $a_n\in  \exp(\fa)$, where 
  $\fa \subeq \fp$ is maximal abelian and $u_n \in K$.
  Passing from $C$ to the larger compact subset
  $\bigcup_{k \in K} k C k^{-1} \subeq K^{se}(0)$, we may
  further assume that $C$ is conjugation invariant under~$K$.
  Then
  \[ p_n c_n p_n^{-1}
    = (u_n a_n u_n^{-1}) c_n (u_n a_n^{-1} u_n^{-1})
    = u_n (a_n c_n'  a_n^{-1}) u_n^{-1} 
  \]
  for $c_n' := u_n^{-1} c_n u_n \in C$.
  Since the sequence $(u_n)$ in the compact group
  $K \cong \Ad(K)$ has a convergent subsequence,
  this reduces our problem to the special case where~$y_n \in \fa$. 

Consider the natural representation of $\fa$ on
  $\End(\g)$ by $y.A := [\ad y,A]$. We thus obtain a weight space
  decomposition
  \[ \End(\g) =  \bigoplus_{\alpha \in \fa^*} \End(\g)_\alpha.\]
  Accordingly, we write
  \[ c_n = \sum_\alpha c_n^\alpha \in \Ad(K) \subeq \End(\g).\]
  We likewise write $c= \sum_\alpha c^\alpha$
  and
  \[ \supp(c) := \{ \alpha \in \fa^* \: c^\alpha \not=0\}.\] 
  As $c \in K^{se}(0)$, its centralizer is compactly embedded
  (Lemma~\ref{lem:cent-in-k} and Theorem~\ref{thm:compemb}),
  so that $\supp(c)$ spans $\fa^*$. 
  If $\alpha \in \supp(c)$, then there exists an endomorphism
  $d^{-\alpha} \in \End(\g)$ with
  $\tr(c^\alpha d^{-\alpha}) \not=0$ because the trace form
  $(A,B) \mapsto \tr(AB)$ is invariant under the $\fa$-action.
  As $c_n^\alpha \to c^\alpha$, we then obtain
$\tr(c_n^\alpha d^{-\alpha}) \not=0$ 
    for $n$ sufficiently large. Then
    \[\tr(\exp(y_n) c_n \exp(-y_n) d^{-\alpha}) 
= \tr(c_n \exp(-y_n) d^{-\alpha}\exp(y_n)) 
= e^{\alpha(y_n)} \tr(c_n^\alpha d^{-\alpha}),\]
and since $\tr(c_n^\alpha d^{-\alpha}) \to \tr(c^\alpha d^{-\alpha}) \not=0$,
the sequence $e^{\alpha(y_n)}  \in \R$ converges,
hence $\alpha(y_n)$ is bounded from below for each $\alpha\in \supp(c)$. 

Let $\theta$ be a Cartan involution of $\g$, so that
$\g = \fk \oplus \fp$ is the corresponding eigenspace decomposition. 
Then 
\[ \theta(\exp(y_n) c_n \exp(-y_n))
  = \exp(-y_n) c_n \exp(y_n) \to \theta(g),\]
so that the same argument applies with $y_n$ replaces by $-y_n$.
Therefore the sequence $(\alpha(y_n))_{n \in \N}$ is bounded
for each $\alpha \in \supp(c)$, and since $\supp(c)$ spans
$\fa^*$, the sequence $(y_n)_{n\in \N}$ in $\fa$ is bounded.
We may therefore assume that $y_n \to y$ in $\fa$.
This completes the proof.
\end{prf}

\begin{rem} Suppose that $\g$ is simple hermitian
  and that $G$ is simply connected.
  Then
  \[ K \cong Z(K)_e \times K' \cong \R \times K',\]
  where $K'$ is compact semisimple.
  Choosing $z_0 \in \fz(\fk)$ in such a way that $\Delta^+_p(z_0) = \{i\}$,
  we obtain the above diffeomorphism by
  \[ \R \times K' \to K, \quad (t,k) \mapsto \exp(tz_0)k.\]

  An element $k = \exp(tz_0)k' \in K$ is contained in $K^{se}$ if and only
  if its action on $\fp$ has no non-zero fixed point. This is equivalent
  to
  \[ e^{-it} \not\in \Spec(\Ad_{\fp_\C^+}(k')).\]
  For $x = t z_0 + x' \in \ft$ with $x' \in \ft \cap \fk'$, we have
  \[ \alpha(x) = i t + \alpha(x') \quad \mbox{ for } \quad \alpha \in \Delta_p^+.\]

  The convex subset $\ft^{se}(0)$ is $\cW_\fk$-invariant, hence also invariant
  under the fixed point projection $p_\fz \: \ft \to \fz(\fk)$ with kernel
  $\ft \cap \fk'$. We conclude that
  \[ \ft^{se}(0) = \{ t z_0 + x' \:
    0 < t < 2\pi, (\forall \alpha \in \Delta_p^+)\
    i \alpha(x') \in (-t, 2\pi - t)\}.\]

  We also observe that the point symmetry
\[ s(x) := 2\pi z_0 - x \]
    with respect to $\pi z_0$   leaves $\ft^{se}(0)$ invariant. 
\end{rem}

\section{Elliptic domains and compact order intervals
  in semigroups} 
\mlabel{sec:4}

In this section we discuss aspects relating
  elliptic domains with invariant cones in a
  simple hermitian Lie algebra $\g$.
  Recall that a
  simple real Lie algebra $\g$ possesses a pointed generating closed
  convex cone $W \subeq \g$ invariant under $\Inn(\g)$ if and only
  if $\g$ is hermitian (\cite{HHL89, Ne99}). In this case there exists a
  maximal such cone $W^{\rm max}_\g$, unique up to sign, with the
  property that any other pointed generating invariant cone $W$
  is either contained in $ W^{\rm max}_\g$ or $- W^{\rm max}_\g$. Here
  we are interested only in the maximal cone $W := W_\g^{\rm max}$
  and the corresponding
  biinvariant causal structure on $G$, defined by the cone field
  $(g.W)_{g \in G}$. If $G$ is simply connected,
  then this causal structure corresponds to a global order on
  $G$ if and only if $G$ is of {\it tube type},
    i.e., if the corresponding Riemannian symmetric 
space $G/K$ is biholomorphic to a tube domain $T_\Omega = \Omega + i V$, 
where $\Omega$ is an open symmetric cone in the real vector
space~$V$ (actually a euclidean Jordan algebra) 
(\cite{Ol82}, \cite{Ne90}, \cite[Thm.~VIII.12]{Ne93}).
Then $S_W := \oline{\la \exp W \ra} \subeq G$ is a proper closed
subsemigroup. 
The main result of this section is Theorem~\ref{thm:4.x} which
characterizes $G^{se}(0)$ as the interior of the set
$\comp(S_W)$ of all those elements $s \in S_W$ for which the order interval
$S_W \cap s S_W^{-1}$ is compact. As a consequence, we obtain
that the causal structure on $G^{se}(0)$ is globally
hyperbolic (Corollary~\ref{cor:globhyp1}).

\subsection{The maximal invariant cone and $\g^{se}(0)$}

From \eqref{eq:fixadg} 
it follows in particular that
$\pi z \in \ft^{se}(0)$ lies in the basic alcove.
Moreover, 
\begin{equation}
  \label{eq:om2}
 \ft^{se}(0) - \pi z
 = \{ x \in \ft \: (\forall \alpha \in \Delta_p)\  |i\alpha(x)| < \pi \}.
\end{equation}
We also observe that 
\begin{equation}
  \label{eq:om3}
  \ft^{se}(0)  = (W^{\rm max}_\ft)^\circ \cap (2\pi z - (W^{\rm max}_\ft)^\circ).
\end{equation}
As $\ft^{se}(0)$ is $\cW_\fk$-invariant, we have
\[ \g^{se}(0) \cap \ft = \ft^{se}(0).\]

Let $W := W^{\rm max}_\g \subeq \g$ be the unique pointed generating
invariant convex cone with $W \cap \ft = W^{\rm max}_\ft$.
Then 
\begin{equation}
  \label{eq:Wse1}
  W^{se} := W \cap \g^{se} \subeq \comp(\g)
\end{equation}
(cf.\ Lemma~\ref{lem:c.1a}(iv))
is an open $\Ad(G)$-invariant subset of $W$ that is
determined by $W^{se}_\ft := W^{se} \cap \ft$ via
\begin{equation}
  \label{eq:Wse}
  W^{se}= \Ad(G).W^{se}_\ft = \dot\bigcup_{\vartheta \geq 0} \g^{se}(\vartheta).
\end{equation}
Here 
\[ \g^{se}(0) = \Ad(G)\ft^{se}(0) \subeq W^{se} \]
is the connected component of $W^{se}$ containing $(0,2\pi)z$, 
and
\[ G^{se}(0) = \exp(\g^{se}(0)) \]
is the connected component of $G^{se}$ containing $\exp(z)$
(Theorem~\ref{thm:conncomp}).
We also note that
\begin{equation}
  \label{eq:comp-shift}
 \exp(2\pi m z) G^{se}(\vartheta)
 = G^{se}(\vartheta + m)
\end{equation}
follows from $\exp(2\pi z) \in Z(G)$ and \eqref{eq:alc-trans}.

\begin{rem} For the cone $W_\fk := W_\g^{\rm max} \cap \fk$, we have 
  \[ \fk^{se}(0) := \fk \cap \g^{se}(0) = \Ad(K) \ft^{se}(0)
    \  {\buildrel !\over  =}
   \ W_\fk^\circ \cap \big(2\pi z - W_\fk^\circ\big),\] 
which  follows from the $\Ad(K)$-invariance of both sides by reduction to
 \eqref{eq:om3}. This is an open order interval
 for the order on $\fk$ defined by the pointed generating closed
convex  cone $W_\fk$. In particular,  $\fk^{se}(0)$ is an open subset whose
 boundary is contained in
 \[ \partial W_\fk \cup \big(2\pi z - \partial W_\fk\big).\]
 \end{rem}

 \begin{rem} Let $\cD \subeq \fp_\C^+ \into G_\C/K_\C P_\C^-$
   denote the realization of the Riemannian symmetric space
   $G/K$ as a bounded symmetric domain.

   The subset $G^e = \bigcup_{g \in G} g K g^{-1}$ consists of all elements
   $g \in G$ for which $\Fix(g) \subeq \cD$ is non-empty.
   In this characterization $g \in G^{se}$ is equivalent to
   $|\Fix(g)| = 1$, because this is equivalent to $Z_G(g)$ being
   compactly embedded (Theorem~\ref{thm:compemb}), resp., to the action of $g$ on the tangent
   space $T_m(\cD)$ of any fixed point $m$ to have no fixed vectors.

   This leads to a natural map
   \[ \Psi \: G^{se} \to \cD, \quad  \Fix(g) = \{ \Psi(g)\}, \]
   which is equivariant with respect to the conjugation action on $G$:
   \[ \Psi(g_1 g_2 g_1^{-1}) = g_1. \Psi(g_2).\]
   In particular, we have for $k \in K^{se}$ and $p = \exp x \in \exp(\fp)$:
   \[ \Psi(pk p^{-1}) = p.\Psi(k) = \exp(x) K = \exp(x).0 \in \cD.\]
   We may also think of $G^{se}(0)$ as a $G$-equivariant fiber bundle
   \[ G^{se}(0) \cong G \times_K K^{se}(0)
   = (G \times K^{se}(0))/K, \quad
   (g,k_0).k = (gk, k^{-1} k_0 k).\]
 \end{rem}

 \begin{rem} (Connecting with tube domains of causal symmetric spaces)
   In this remark we briefly explain how the domains
   $G^{se}(0)$ occur naturally in the context of causal symmetric spaces
   (cf.\ \cite{NO23}).
      
We consider the noncompactly causal symmetric space
\[ M := G_\C/G \quad \mbox{ with } \quad -i W_\g^{\rm max} \subeq i \g
  \cong T_{e G}(M)\]
defining the causal structure.
We consider $G_\C$  as a complex symmetric space, containing
\[ G = \{ g \in G_\C \:\oline g = g\}_e
\quad \mbox{ and } \quad  M \cong \{ g \in G_\C \:\oline g = g^{-1}\}_e \]
as symmetric subspaces. 
Here we assume that $G = G_\C^\sigma$ for $\sigma(g) = \oline g$,
which holds in particular if $G_\C$ is simply connected.

For $h := -iz \in i\g$, the operator
$\ad h$ is diagonalizable with eigenvalues $\{-1,0,1\}$,  and the domain 
\[ W^{\rm res}_{\ft}
  := \Big\{ x \in W_\ft^{\rm max} \:\Spec\Big(\ad \Big(x - \frac{\pi}{2}z\Big)\Big)
  \subeq \Big(-\frac{\pi}{2},\frac{\pi}{2}\Big) \cdot i\Big \} 
  =  \frac{1}{2}\ft^{se}(0)\]
in $\ft$ defines the $\Ad(G)$-invariant domain 
\[ W^{\rm res} := \Ad(G)W^{\rm res}_{\ft} = \frac{1}{2} \g^{se}(0),\]
for which the exponential function
$\Exp_e(x) = \exp(2x)$ satisfies 
\begin{equation}
  \label{eq:expew}
\Exp_{e}(W^{\rm res})
= \exp(\g^{se}(0)) = G^{se}(0)
\end{equation}
is the connected component $G^{se}(0)$ in the open subset
$G^{se}$ of the symmetric subspace $G \subeq G_\C$. 

The complex tube domain associated to $M$ is 
\[ \cT_M \cong G_\C.G^{se}(0) = \{ g w \oline g^{-1} \:
w \in G^{se}(0)\} \subeq G_\C, \] 
(see \cite{NO23} for details).
So $\cT_M$ is a $G_\C$-equivariant fiber bundle
$\cT_M\to G_\C/G$ whose fiber in the base point is $G^{se}(0)$.
It contains $M \cong G_\C/G$ as a totally real submanifold in its boundary. 
\end{rem}

\begin{lem} \mlabel{lem:3.6} If
  $W = W^{\rm max}_\g$, then $W^{se} = W \cap \g^{se} \subeq W^\circ$.
 \end{lem}

 \begin{prf} Let $x \in W^{se}$. Then its adjoint orbit intersects
   $W^{se}_\ft = \ft^{se} \cap W^{\rm max}_\ft$, so that we may assume that
   $x \in W^{\rm max}_\ft$, i.e., $i\alpha(x) \geq 0$ for all $\alpha \in\Delta_p^+$. If one non-compact root vanishes on $x$,
   then $\ker(\ad x)$ is not compactly embedded, contradicting
   $x \in \g^{se}$. Therefore $x \in (W^{\rm max}_\ft)^\circ \subeq W^\circ$.
 \end{prf}

 \begin{defn} Suppose that the invariant cone $W\subeq \g$ is
   {\it global in $G$}, i.e., that the closed subsemigroup
   \[ S_W := \oline{\la \exp W \ra} \]
   satisfies
   \[ \L(S_W) := \{ x \in \g \: \exp(\R_+ x) \subeq S_W \} = W.\]

 \nin (a) We write $g \preceq_{S_W} h$ if $g^{-1} h \in S_W$
 and $g \prec_{S_W} h$ if $g^{-1} h \in S_W^\circ$.
 For $g \in G$, we put

 \nin (b)  The {\it order intervals} in $G$ are the sets
   \begin{equation}
     \label{eq:oint}[g,h]_{\preceq} := g S_W \cap h S_W^{-1}
     \quad \mbox{ and } \quad
     [g,h]_{\prec} := g S_W^\circ \cap h (S_W^\circ)^{-1}
   \end{equation}
We write 
\[ \comp(S_W) := \{ s \in S_W \: [e,s]_{\preceq} \ \mbox{ compact} \}\]
for the set of all elements $s \in S_W$ for which
 the order interval $[e,s]_{\preceq}$ is compact. 
\end{defn}

The following lemma collects some fundamental information about
$\comp(S_W)$ and its relation with the cone~$W$. 

\begin{lem} \mlabel{lem:connected} 
  The following assertions hold:
  \begin{itemize}
  \item[\rm(a)] The interior $\comp(S_W)^\circ$ of $\comp(S_W)$ in $G$ is connected.    
\item[\rm(b)] $S_W \setminus \comp(S_W)$ is a closed subset of $G$.   
  \item[\rm(c)] $\comp(S_W) \cap S_W^\circ = \comp(S_W)^\circ$.   
  \item[\rm(d)]    $\comp(S_W) \subeq \exp(W)$. 
  \item[\rm(e)] For each $s \in S \setminus \comp(S_W)$
    there exists $0\not= w \in W$ with $\exp(\R_+ w) \subeq s S_W^{-1}$. 
\item[\rm(f)] $\exp(W^\circ) \subeq S_W^\circ$.     
\item[\rm(g)] $\comp(S_W)^\circ \subeq G^{se}(0)$.
\item[\rm(h)] $\exp^{-1}(\comp(S_W)) \cap W^\circ \subeq \g^{se}(0)$.
  \end{itemize}
 \end{lem}

 \begin{prf} (a) Let $a,b \in \comp(S_W)^\circ$. Then
$e \in [a,b]_{\prec}$ 
   implies the existence of a $c \in S_W^\circ$, contained in the intersection
   on the right. Then there are causal paths from $c$ to $a$ and $b$,
   which are contained in $S_W^\circ \cap a S_W^{-1}\subeq \comp(S_W)$
   and $S_W^\circ \cap bS_W^{-1} \subeq \comp(S_W)$, respectively.
   The concatenation of these paths connects $a$ and $b$ in
   $\comp(S_W)^\circ$.

\nin (b) follows from \cite[\S 1]{Ne91}.

\nin (c)     As $S_W \setminus \comp(S_W)$ is closed in $G$ by (b), 
  the intersection
  $\comp(S_W) \cap S_W^\circ$ is open in $G$, hence contained in
  $\comp(S_W)^\circ$. The converse inclusion is clear.

  \nin (d) follows from \cite[Thm.~IV.6]{MN95}.

   \nin (e)    Since the invariant cone $W = \L(S_W)$ is pointed,
   this follows from \cite[Cor.~VI.3]{Ne93}.

   \nin (f) Since  $\exp$ is a local diffeomorphism 
   around $0$, for $x \in W^\circ$, we have
   $\exp(tx) \in (\exp W)^\circ \subeq S_W^\circ$ for $t$ sufficiently small. As
   $S_W^\circ$ is a semigroup ideal in $S_W$, (f) follows.

   \nin (g)  Let $s \in \comp(S_W)^\circ$.
  Then the order interval $C := [e,s]_{\preceq}$ is compact
  and invariant under $Z_G(s)$. Therefore $CC^{-1}$ is a compact $e$-neighborhood
  invariant under the centralizer $Z_G(s)$.
  For every $g \in C$, the conjugation orbit under the group
  $Z_G(s)$ is relatively
  compact.   Now $G = \bigcup_{n \in \N} (CC^{-1})^n$ implies that,
  for each $g \in G$, the conjugation orbit under $Z_G(s)$ is relatively
  compact. So Theorem~\ref{thm:compchar}(c) shows that every
  $x \in  \L(Z_G(s)) = \g^{\Ad(s)}$ is compactly embedded, hence that
  $s \in G^{se}$ follows by applying Lemma~\ref{lem:algrp} to
  $\phi = \Ad(s)$. 

  Next we observe that $\comp(S_W)^\circ = U \cap S_W^\circ$
  for a sufficiently small $e$-neighborhood $U \subeq G$
  (\cite[Cor.~III.6]{MN95}), and thus
  $G^{se}(0) \cap \comp(S_W)^\circ \not=\eset$.
  Now the connectedness of $\comp(S_W)^\circ$ (Lemma~\ref{lem:connected})
  leads to~(g). 

  \nin (h) Let $x \in W^\circ$ with $\exp(x) \in \comp(S_W)$.
  Then (f) implies that $\exp(x) \in S_W^\circ$, hence by (c) further
  $\exp(x) \in \comp(S_W)^\circ$ and thus $\exp(x) \in G^{se}(0)$ by~(g).
  By Definition~\ref{def:1.1}, this means that $x \in \g^{se}$.
  Applying a suitable $\Ad(g)$, we may assume that $x \in \ft$.
    Since the same conclusions apply to the 
  multiples $tx$, $0 < t \leq 1$, we obtain
$0 < i\alpha(x) < 2\pi$ for every $\alpha \in \Delta_p^+,$ 
i.e., $x \in \ft^{se}(0) \subeq \g^{se}(0)$. We thus obtain~(h).
\end{prf}

\begin{lem} \mlabel{lem:algrp}
  Let $\g$ be a semisimple real Lie algebra.
  If $\phi \in \Aut(\g)$ is such that $\Fix(\phi) = \g^\phi$
  is compactly embedded, then $\phi$ is elliptic in $\Aut(\g)$.
\end{lem}

\begin{prf} We consider the mutliplicative Jordan decomposition
  $\phi = \phi_s \phi_u$ into commuting semisimple and unipotent
  factors in the algebraic group $\Aut(\g)$ (\cite[Ch.~3, \S 2, Thm.~6]{OV90}).

  First we show that $\phi$ is semisimple, i.e., that $\phi_u = \id_\g$.
  If this is not the case, then $D_n := \log \phi_u$ is a derivation
  of $\g$, hence of the form $\ad x_n$ for some nilpotent element
  $x_n \in \g$. Then $\ad x_n$ commutes with $\phi$, which means that
  $x_n \in \Fix(\phi)$. But this subalgebra is compactly embedded,
  so that $x_n$ is also elliptic and therefore $x_n = 0$.

  This shows that $\phi$ is semisimple, so that
  $\g_\C$ decomposes into $\phi$-eigenspaces $\g_\C(\lambda)$, satisfying
  \[ [\g_\C(\lambda), \g_\C(\mu)] \subeq \g_\C(\lambda\mu)
  \quad \mbox{ for } \quad \lambda, \mu \in \C^\times.\]
  For $\theta\in \R$, let
  \[ \g_\C^\theta := \sum_{|\lambda| = e^\theta} \g_\C(\lambda).\]
  These subspaces are invariant under complex conjugation, so that
$\g^\theta := \g_\C^\theta \cap \g$ 
  defines a family of real subspace of $\g$ satisfying
  \[ [\g^{\theta_1}, \g^{\theta_2}] \subeq \g^{\theta_1 + \theta_2}.\]
  We thus obtain a diagonalizable derivation $D_h$ of $\g$ by
  \[ D_h x = \theta x \quad \mbox{ for } \quad x \in \g^\theta.\]
  Now $D_h =\ad x_h$ for some hyperbolic element $x_h \in \g$ and,
  as above, it follows that $x_h \in \Fix(\phi)$. Now $x_h$
  is hyperbolic and elliptic, and we arrive at $x_h = 0$.
  This proves that all eigenvalues of $\phi$ have unit length,
  which means that $\phi$ is elliptic in $\Aut(\g)$. 
\end{prf}

\begin{rem} (a) The preceding lemma works with the same argument
  for any Lie algebra $\g$ for which
  $\ad \: \g \to \der(\g)$ is bijective, i.e., $\fz(\g) = \{0\}$
  and all derivations are inner.

  \nin (b) If $\g$ is the $3$-dimensional Heisenberg algebra with the
  basis $p,q,z$ satisfying $[p,q] = z$ and $z$ is central, then
  \[ \phi(p) = 2p, \quad \phi(q) = \frac{1}{2} q, \quad
    \phi(z) = z \]
  is an automorphism for which $\Fix(\phi) = \fz(\g)$ is compactly
  embedded, but $\phi$ is hyperbolic and not elliptic. 
\end{rem}

 \begin{lem} \mlabel{lem:component-incl}
   $G^{se}(0) \subeq [e,\exp(2\pi z)]_{\prec} = S_W^\circ \cap (\exp 2\pi z) (S_W^\circ)^{-1}$.
    \end{lem}

 \begin{prf} First, $\g^{se}(0) \subeq W^\circ$ implies
   that
   \[ G^{se}(0) = \exp(\g^{se}(0)) \subeq \exp(W^\circ) \subeq S_W^\circ,\]
where we have used Lemma~\ref{lem:connected}(f) for the last inclusion.
   We further find with \eqref{eq:invbas-alc} and \eqref{eq:comp-shift}
   that 
   \[ G^{se}(0)^{-1} \exp(2\pi z)
     = G^{se}(-1) \exp(2\pi z) = G^{se}(0),\]
   hence
   \[ G^{se}(0)  = \exp(2\pi z) G^{se}(0)^{-1} \subeq
     \exp(2\pi z)(S_W^\circ)^{-1}.\qedhere\]
\end{prf}

\subsection{The   Guichardet--Wigner quasimorphism} 
\mlabel{subsec:GW}

In this subsection we recall the Guichardet--Wigner
  quasimorphism as an effective tool to study
  the global structure of the simply connected covering groups
  of hermitian Lie groups.

Let $G$ be a $1$-connected simple hermitian Lie group
and $G = KAN$ be an Iwasawa decomposition. Then $K$ is also simply
connected, so that $\exp \: \fz(\fk) \to Z(K)_e$ is a diffeomorphism
and $Z(K)_e \cong \R$. Moreover, the projection $p_{\fz(\fk)} \: \fk \to \fz(\fk)$
along $[\fk,\fk]$, which is a homomorphism of Lie algebras, integrates to a
group homomorphism
\begin{equation}
  \label{eq:pz}
 p_Z \:  K \to \fz(\fk) \cong \R \quad \mbox{ with} \quad
 p_Z(\exp x) = p_{\fz(\fk)}(x) \quad \mbox{ for } \quad x \in \fk.
\end{equation}

\begin{defn} Let $G$ be a group. A function $f \: G \to \R$ is called a
  {\it quasimorphism} if
  \[ D_f := \sup_{g,h \in G} |f(gh)-f(g)-f(h)| < \infty.\]
\end{defn}

We consider the function 
  \[ \chi_{\rm Iwa} \: G = KAN \to Z(K)_e \cong \R, \quad
  \chi_{\rm Iwa}(kan) = p_Z(k) \]
  that restricts on $K$ to the homomorphism $p_Z$ from \eqref{eq:pz}.
  By \cite[\S 3.1]{BH12}, its homogenization
  \[ f_{GW}(g) := \lim_{m \to \infty} \frac{\chi_{\rm Iwa}(g^m)}{m} \]  
  defines a quasimorphism, called the
  {\it Guichardet--Wigner quasimorphism} 
  (see also \cite{GW78}, \cite{Sh06}).
  By \cite[\S 2.2.3]{Ca09}, homogeneous quasimorphisms are
  class functions.
  As the restrictions 
    of $\chi_{\rm Iwa}$ to the subgroups $AN$ and $K$ are homomorphisms,
    it coincides with     $f$ on these subgroups:
    \begin{equation}
      \label{eq:qh-hom}
      f_{GW}\res_K = p_Z \: K \to Z(K)_e \cong \R       \quad \mbox{ and } \quad
      f_{GW}\res_{AN} = 0.
    \end{equation}
 Further,
    $\Ad(G)\fn$ is the set of nilpotent elements
    of $\g$ , $\Ad(G)\fa$ the set of hyperbolic elements and
    $\Ad(G)\fk$ the set of elliptic elements (see \cite[\S IX.7]{Hel78}).
    Moreover,
    $\fn + \fa$ is a maximal triangulizable subalgebra and thus meets
    the adjoint orbit of every element with real spectrum.

    \begin{prop} \mlabel{prop:fgw2}
Let $x \in \g$ with the real Jordan decomposition
    \[ x = x_e + x_h + x_n, \]
    with $x_e$ elliptic, $x_h$ hyperbolic  and $x_n$ nilpotent
    and pairwise commuting.
    Then the quasimorphism $f_{GW}$ has the following properties:
\begin{itemize}
\item[\rm(a)] If $x_e = 0$, i.e., $x$ has real spectrum, 
      then $f_{\rm GW}(\exp x) = 0$.
\item[\rm(b)] If $x = x_e$ is elliptic and
  $\Ad(g)x \in \fk$, then $f_{\rm GW}(\exp x) = p_{\fz(\fk)}(\Ad(g)x)$.
\item[\rm(c)] On $G^{se}(0)$ it is the unique smooth
      conjugation invariant function that coincided with
      $p_Z$ on $K^{se}(0)$.
\end{itemize}
\end{prop}

\begin{prf} The conjugacy
    invariance of $f_{GW}$ permits us to evaluate it on one-parameter
    groups generated by elements $x_h + x_n$ with real spectrum 
    and elliptic elements.

    \nin (a)  If $x = x_h + x_n\in \g$, then $f_{GW}(\exp x) = 0$
    follows from the fact that $x \in \Ad(G)(\fa + \fn)$
    (\cite[Thm.~15.14]{Bo91}) because
    $f_{GW}$ is conjugation invariant and vanishes on
    $\exp(\fa + \fn) = AN$. 

    \nin (b) follows directly from the conjugation invariance of $f_{GW}$.

    \nin (c) On $G^{se}(0)$ the function $f_{GW}$ is constant on
    conjugacy classes and coincides with $p_Z$ on
    $K^{se}(0) = G^{se}(0) \cap K$. 
      As $p_Z$ is smooth, 
      Proposition~\ref{prop:2.11} shows that $f_{GW}$ is smooth on $G^{se}(0)$.
\end{prf}

\begin{rem} (a) If $\zeta \in Z(G) \subeq K$, then
  $\chi_{\rm Iwa}((g\zeta)^n) = \chi_{\rm Iwa}(\zeta^n g^n) = n\chi_{\rm Iwa}(\zeta) + \chi_{\rm Iwa}(g^n)$, so that
  we also get
  \[ f_{GW}(g\zeta) = \chi_{\rm Iwa}(\zeta) + f_{GW}(g).\]

  \nin (b) The observation under (a) permits to calculate $f_{GW}$ for
  $G := \tilde\SL_2(\R)$. Let us call $g \in G$
  {\it elliptic, hyperbolic, unipotent} if $\Ad(g)$
  has the   corresponding property. Then an elementary discussion
  using the real Jordan normal form of elements of $\SL_2(\R)$ implies
  that every element of $G$ is elliptic, hyperbolic or unipotent.
  It is elliptic if and only if it is conjugate to an element $g' \in K$,
  and then
  \[ f_{GW}(g) = f_{GW}(g') = p_Z(g').\]
  If it is hyperbolic, it is conjugate to an element in
  $a\zeta \in A Z(G)$, and then
  \[ f_{GW}(a\zeta) = p_Z(\zeta).\]
  If it is unipotent, it is conjugate to an element of
  $N Z(G)$, and we also have
  \[ f_{GW}(n\zeta) = p_Z(\zeta).\]
\end{rem}

\begin{lem} \mlabel{lem:powermonotone}
  Let $S$ be an invariant subsemigroup of a group $G$.
We consider on $G$ the biinvariant order relation defined by
  \[ g_1 \preceq_S g_2 \quad \mbox{ if } \quad g_1^{-1} g_2 \in S.\]
  Then the following assertions hold:
  \begin{itemize}
  \item[\rm(a)] $a \preceq_S b$ and $c \preceq_S d$ imply 
    $ac \preceq_S bd$. 
  \item[\rm(b)] The power maps $p_n \: G \to G, g \mapsto g^n$,
    are monotone. 
  \end{itemize}
\end{lem}

\begin{prf} (a) From $b \in a S$ and $d \in cS$, we derive that
  $bd \in a S c S = a c S S \subeq ac S,$   hence $ac\preceq_S bd$.

  \nin (b) Inductively we derived from (a) that
$a_j \preceq_S b_j$ for $j = 1,\ldots, n$,   implies that
  \[ a_1 \cdots a_n \preceq_S b_1 \cdots b_n\] 
  and this implies (b).
\end{prf}

If $G$ is  simply connected, 
then the maximal invariant cone $W = W_\g^{\rm max}$
is global in $G$ if  and only if $\g$ is of {\it tube type}, 
i.e., if the corresponding Riemannian symmetric 
space $G/K$ is biholomorphic to a tube domain $T_\Omega = \Omega + i V$, 
where $\Omega$ is an open symmetric cone in the real vector
space~$V$ (actually a euclidean Jordan algebra) (\cite{Ol82}, \cite{Ne90},
\cite[Thm.~VIII.12]{Ne93}).
If $\g$ is of tube type,
this can be derived from the monotonicity of the function
  \[  \chi_{\rm Iwa} \: KAN \to Z(K)_e \cong \R, \quad  kan \mapsto p_{Z}(k). \]

  \begin{prop} \mlabel{prop:tt-fgw}
    For a hermitian Lie algebra $\g$ of tube type,
    the Guichardet--Wigner quasimorphism $f_{GW}$ has the following
  properties:
  \begin{itemize}
  \item[\rm(a)] It defines a monotone function
    $f_{GW} : (G, \preceq_{S_W}) \to \R$.
  \item[\rm(b)] If $0 \not=x \in W$  is not nilpotent,
    then there exist $C > 0$ such that 
      $f_{GW}(\exp(tx)) \geq t C$ for $t \in \R$.
  \end{itemize}
\end{prop}

\begin{prf}
  (a) follows from the fact it is a limit of the monotone functions
  $g \mapsto \frac{\chi_{\rm Iwa}(g^m)}{m}$ whose monotonicity follows from
  Lemma~\ref{lem:powermonotone} and the monotonicity of $\chi_{\rm Iwa}$.

  \nin (b) By (a) the function $f_{GW}$
  is monotone on all curves $t \mapsto \exp(tx)$ for $x \in W$. 
For $x \in \partial W$, we have a Jordan decomposition
$x = x_e + x_n$, where $x_e$ is elliptic and $x_n$ nilpotent,
both contained in~$W$.
That $x_e$ and $x_n$ are both contained in $W$ follows from 
\[ x_e + \R_+ x_n  \subeq \Ad(G)x \subeq W\]
(\cite[Cor.~B.2]{NOe22}). 
With the monotonicity of $f_{GW}$ this implies that 
\[   f_{GW}(\exp(t x))
  = f_{GW}(\exp(t x_e) \exp(t x_n))
\geq f_{GW}(\exp(t x_e)).\] 
As $x_e$ is conjugate to some $y_e \in W \cap \fk$, 
\[ f_{GW}(\exp tx_e) = f_{GW}(\exp ty_e)
  = p_Z(\exp t y_e) = t p_{\fz(\fk)}(y_e), \]
and this proves (b).
\end{prf}

\begin{ex}
  Let $G=\widetilde\Sp_{2n}(\R)$ be the simply connected covering of
  $\Sp_{2n}(\R)$. In this case  $K\cong \widetilde\U_n(\C)$
  and the Guichardet--Wigner quasimorphism $f_{GW}$ coincides with the well-known Maslov quasimorphism, which was first considered in \cite{GL58} and plays a key role in the study of Hamiltonian systems and their stability, see e.g. \cite{SZ92}. It is the unique homogeneous quasimorphism on $\widetilde{\Sp}_{2n}(\R)$ that coincides on $\widetilde\U_n(\C)$
  with the lift of the complex determinant \cite{BS10}.
Under the natural embedding of $\widetilde\Sp_{2n}(\R)$
into the universal cover of the identity component of the group $\mathrm{Cont}_0(\R P^{2n-1})$ of diffeomorphisms of $\R P^{2n-1}$ preserving the standard contact structure, the Maslov quasimorphism extends
to a homogenous quasimorphism on $\widetilde\Cont_0(\R P^{2n-1})$,
known as Givental's asymptotic non-linear Maslov index.
\end{ex}

\subsection{The causal  flag manifold} 

Let $\fa \subeq \fp$ be maximal abelian and
pick $h \in \fa$, so that its eigenspaces
$\g_\lambda := \g_\lambda(h)$, $\lambda = -1,0,1$, define a $3$-grading of $\g$.
By \cite[Prop.~3.11(b)]{MN21}, such an element exists if and only if $\g$
is of tube type. 
We write
\[  G_{\pm 1} := \exp(\g_{\pm 1}) \quad \mbox{ and } \quad
  P^\pm := G^h G_{\pm 1}, \quad
  G^h = Z_G(h) =\{ g\in G \: \Ad(g)h = h\}.\]
Then
\[  M := G/P^- \]
is a compact homogeneous space
with a natural $G$-invariant causal structure specified by
the cone
\[ C_+ := W \cap \g_1 \subeq \g_1 \cong T_{m_0}(M), \quad \mbox{ where }\quad
  m_0 := eP^- \in M \]
is the base point. Identifying the tangent space in $m_0$ with
$\g/\fp^- \cong \g_1$, the cone $C_+ = W \cap \g_1 = p_{\g_1}(W)$ 
(\cite[Lemma~3.2]{NOO21}) is invariant under the $P^-$-action, hence defines a
$G$-invariant causal structure on~$M$.

As $M$ is compact, the causal structure on $M$ possesses
  closed causal curves, but this changes if we consider its simply
  connected covering $\tilde M := G/P^-_e$. The order on $\tilde M$
  is specified by the Lie subsemigroup
  \[ S_M := \oline{S_W P^-_e} = \oline{\la \exp(W + \fp^-)\ra}.\]

\begin{prop} $S_M \cap \exp(\g_1) = \exp(C_+)$.     
\end{prop}

\begin{prf} The subset
  $S_1 := \{x \in \g_1 \:  \exp(x) \in S_M\}$ is a closed subsemigroup
  of the additive group  $\g_1$ containing the closed convex cone $C_+$
  and invariant under the group $G^h_e$.

  Let $x \in \g_1 \setminus C_+$.
  If $(c_1, \ldots, c_r)$ is a Jordan frame in $\g_1$,
    the Spectral Theorem for Euclidean Jordan Algebras
    \cite[Thm.~III.1.3]{FK94} implies that the
    orbit of $x$ contains an element of the form
  \[ x' := c_1 + \cdots + c_k  - c_{m+1} - \cdots - c_r, \quad
    1 \leq k \leq m < r\]
  (see also \cite[Cor.~IV.2.7]{FK94} for the transitivity if
  $K^h_e$ on the set of Jordan frames).
If $x \in - C_+$, then the
  subsemigroup generated by $x$ and $C_+$ contains the cone
  $C_+ + \N x = C_+ + \R x$, which
  is excluded. Therefore $x$ has at least one positive spectral value,
  which means that $k \geq 1$. Further, $m < r$ corresponds  to
  $x \not\in C_+$. 
  We assume that the number $r- m$ of negative spectral values is maximal.
  The left multiplication $L(x)y := xy$ on the Jordan algebra then satisfies
  $e^{L(x)} \in \Ad(G^h_e)$ and
  \[ e^{\log(2)(L(c_r) - L(c_1))}x'
    = \frac{1}{2} c_1 + \cdots + c_k  - c_{m+1} - \cdots - 2 c_r.\]     
Therefore  $x'$ is conjugate to     
 \[ x'' := -2 c_1 + c_2 + \cdots + c_k
   - c_{m_1} - \cdots - c_{r-1} + \frac{1}{2} c_r,\]
 so that
 \[ x' + x'' := - c_1 + 2c_2 + \cdots + 2c_k
   - 2c_{m_1} - \cdots - 2c_{r-1} - \frac{1}{2} c_r\]
 has $r - m + 1$ negative spectral values, contradicting the maximality
 of $r-m$. This contradiction implies that $S_1 = C_+$,
 so that $S_M \cap G_1 = \exp(C_+)$.
\end{prf}

\begin{cor} \mlabel{cor:ord-emb} The open embedding
  \[ (\g_1, \preceq_{C_+}) \into (\tilde M, \preceq), \quad x \mapsto
  \exp(x) P^- \]
 is an order embedding. 
\end{cor}

\begin{lem} \mlabel{lem:GM-mon} If $\gamma \: [a,b] \to G$ is a $W$-causal curve
  and $m \in \tilde M$, then the curve
  $\gamma_M(t) := \gamma(t).m$ is causal in $\tilde M$.
\end{lem}

\begin{prf} We write $m = g_0 P^-_e$, $m_0 = eP^-_e$ and consider the
  map $F \:  G \to \tilde M, g \mapsto g.m = gg_0 P^-_e$.
  We have to show that $F$ is monotone. This follows from 
  \[ \dd F(g)(g.W) 
    = (g.W).g_0.0_{m_0} 
    = (gg_0.W).0_{m_0}  = gg_0.p_{\g_1}(W) = gg_0.C_+.
    \qedhere\]   
\end{prf}

\subsection{Connecting $G^{se}(0)$ with compact order intervals}

For $t > 0$, we consider the elements $\exp(tz) \in S_W^\circ$.
Next we determine for which values of $t$ we have
$\exp(tz) \in \comp(S_W)$. By Lemma~\ref{lem:connected}(b),
there exists a $T > 0$ such that this is the case
for $t < T$ and never for $t \geq T$.
We now show that $T = 2\pi$.  We first take a closer look
at the case $\g = \fsl_2(\R)$. 

\begin{ex} \mlabel{ex:sl2}
  For $\g = \fsl_2(\R)$, we consider the natural $3$-grading defined by
  the element
  \[ h := \frac{1}{2} \pmat{1 & 0 \\ 0 & -1} \in
    \fa = \R h \quad \mbox{ and } \quad
    z := \frac{1}{2} \pmat{0 & 1 \\-1 & 0}.\]
  Then 
  \[ \fk = \fz(\fk) = \R z = \so_2(\R) \quad \mbox{ and } \quad
    \g_1 = \R \pmat{0 & 1 \\ 0 & 0}= \fn, \quad 
    \g_0 = \fa = \R h, \quad 
    \g_{-1} = \R \pmat{0 & 0 \\ 1 & 0}.\]
  The action of $\SL_2(\R)$ on $\R_\infty \cong \bS^1$ by M\"obius transformations
  \[  \pmat{ a & b \\ c & d}.x = \frac{a x + b}{c x + d}\]
  then implies that
  \[ M = G/P^- \cong \R_\infty \cong \bS^1 \quad \mbox{ and }\quad
    \tilde M \cong \tilde \R_\infty \cong \R.\]
  Here we use that $\fa + \g_{-1} = \fp^-$
  is the Lie algebra of the stabilizer of $0$. 
We consider the corresponding action of 
$G = \tilde\SL_2(\R)$ on $M = \R_\infty$.
Then
\[  \exp(tz).x = \pmat{ \cos(t/2) & \sin(t/2)  \\
    -\sin(t/2) & \cos(t/2)}.x = \frac{\cos(t/2) x + \sin(t/2)}
  {-\sin(t/2)x + \cos(t/2)}\]
with
\begin{equation}
  \label{eq:tan1}
  \exp(tz).0 = \tan(t/2).
\end{equation}
Hence $g \in \exp((-\pi,\pi)z) AN$ implies that $g.0 \in \R$,
and therefore
\[ \exp(tz).P^- \subeq \exp(\g_1)P^- \quad \mbox{ for } \quad |t| < \pi.\]
With this observation, we can also evaluate $\chi_{\rm Iwa}$ to 
\[ \chi_{\rm Iwa}(g) = 2\arctan(g.0).\]
For $g = \exp(sw), w = \pmat{0 & 1 \\ 0 & 0}$, this leads with
$g.0 = s$ to
\[ \chi_{\rm Iwa}(\exp(sw)) = 2 \arctan(s).\]
We therefore derive from $\exp(\R_+ w) \subeq [e,\exp(tz)]$ that
\[  2 \arctan(\alpha) \leq t - 2 \arctan(\beta) \quad
  \mbox{ for }\quad \alpha, \beta \geq 0,\]
and this implies that $t \geq 2 \pi$.

Alternatively, we may observe that, for all $s > 0$, we have 
\[ \exp(sw).\exp\big(-\frac{t}{2}z\big).0
  = s - \tan(t/4) \leq
  \exp\Big(\frac{t}{2}z\Big).0 = \tan(t/4),\]
which is absurd.
\end{ex}

\begin{prop} \mlabel{prop:4.14} If $\g$ is simple hermitian of tube type
  and $W = W_\g^{\rm max}$, then
    $\exp(tz) \in \comp(S_W)$ if and only if
  $0 < t < 2\pi$.   
\end{prop}

\begin{prf} If $\exp(tz) \in \comp(S_W)$, then
  $\exp(sz) \in \comp(S_W)^\circ$ for $0 < s < t$
  by Lemma~\ref{lem:connected}(c).  Lemma~\ref{lem:connected}(g)
  now implies that $\exp(sz) \in G^{se}(0)$ for $0 < s < t$, and hence that
  $t \leq 2\pi$ (Theorem~\ref{thm:compemb}). 

  It remains to show that, if $\exp(tz) \not\in \comp(S_W)$,
  then $t \geq 2\pi$. With Lemma~\ref{lem:connected}(e), we
  find a non-zero $w \in W$ with $\exp(\R_+ w) \subeq [e, \exp(tz)]$.
  Let $f_{GW}$ be the Guichardet--Wigner quasimorphism  corresponding
  to the Cartan decomposition $\g = \fk \oplus \fp$.
  This function is $W$-monotone 
  by Proposition~\ref{prop:tt-fgw}(a), so that 
  \[ f_{GW}(\exp sw) \leq f_{GW}(\exp tz) = t
    \quad \mbox{ for } \quad s \in \R_+,\]
  and this implies that $w$ is nilpotent (Proposition~\ref{prop:tt-fgw}(b)).

Now \cite[Prop.~III.4]{HNO94} implies the existence
  of a subalgebra $\fs \cong \fsl_2(\R)$, $r = \rank_\R(\g)$,
  such that $w \in \fs$, the subalgebra $\fs$ is invariant
  under the Cartan involution, and $\fa \subeq \fs$.
  It follows in particular that $\fs$ is also adapted to the
  $3$-grading, which is defined by an element $h\in \fa$.
  In addition, \cite[Thm.~II.10]{HNO94} yields a
  subalgebra $\fs_0 \subeq \fs$, also containing $w$,
  which is also adapted to the Cartan involution and the $3$-grading.

Then $W_\fs := W \cap \fs$ is a generating invariant cone in $\fs$.
We write $\fs_1, \ldots, \fs_r$ for the simple ideals of $\fs$
which are also adapted to the $3$-grading. 
  As $\g$ is assumed to be of tube type, we even have
    $\R z = \fz(\fk) \subeq \fs$ (\cite[Prop.~5.2]{MNO23}). Then
  \[ z = (z_1, \ldots, z_r) \quad \mbox{ with } \quad
    \fz(\fk \cap \fs_j) = \R z_j \quad \mbox{ and } \quad
    z_j \in W_{\fs_j}.\]
  To verify that 
  $\exp(tz).P^- \in \exp(\g_1)$
for $|t| < \pi$,  it suffices to verify
  the corresponding assertion for $z_j \in \fs_j$,
  but in this case it follows from Example~\ref{ex:sl2}.

  Suppose that $t_0 := \frac{1}{2} t < \pi$.
  Then $\exp(\pm t_0 z).P^- = \exp(x_{\pm}).P^-$ for
  $x_\pm \in \g_1$ (Example~\ref{ex:sl2}). 
  Further, $\exp(\R_+ w) \exp(-t_0 z) \prec_S \exp(t_0 z)$ implies with
Corollary~\ref{cor:ord-emb} that
  \[ \exp(\R_+ w) \exp(-t_0 z).P^-_e
    = \exp(\R_+ w + x_-).P^-_e \subeq 
    \exp(x_+ - C_+).P^-_e.\]
  This leads to
  \[ \R_+ w + x_- \subeq x_+ - C_+,\]
  which contradicts the pointedness of the cone $C_+$. 
  We conclude that $t \geq 2\pi$.
\end{prf}

\begin{thm}
  \mlabel{thm:4.x}
  If $\g$ is simple hermitian of tube type and
  $W = W_\g^{\rm max}$, then 
  \[  \comp(S_W)^\circ  = G^{se}(0) = S_W^\circ \cap \exp(2\pi z)
    (S_W^\circ )^{-1}.\] 
\end{thm}

\begin{prf} In view of
  Lemma~\ref{lem:connected}(g) and   Lemma~\ref{lem:component-incl},
  it suffices to show that
  any $s \in S_W^\circ \cap \exp(2\pi z) (S_W^\circ)^{-1}$
  is contained in $\comp(S_W)^\circ$. 
  Pick an $\eps > 0$ with
  $s \in \exp((2\pi - \eps)z)(S_W^\circ)^{-1}$ 
  and recall from Proposition~\ref{prop:4.14} that 
$\exp((2\pi - \eps)z) \in \comp(S_W)$. 
  Then the assertion follows from
  \[ s \in \comp(S_W) \cap S_W^\circ 
    = \comp(S_W)^\circ \]
  (Lemma~\ref{lem:connected}(c)).   
\end{prf}

\begin{cor} \mlabel{cor:globhyp1} If $\g$ is simple hermitian of tube type, then
  $G^{se}(0)$ is globally hyperbolic.
\end{cor}

\begin{prf} For $a, b \in G^{se}(0) = \comp(S_W)^\circ$ (Theorem~\ref{thm:4.x}),
  we have
  \[ [a,b] \subeq S_W^\circ \cap \comp(S_W)
    \ {\buildrel {\ref{lem:connected}(c)}\over =}\ \comp(S_W)^\circ = G^{se}(0).\] 
  This implies in particular that the order intervals for the causal
  order in $G^{se}(0)$ coincide with the order intervals specified
  by the $S_W$-order on $\tilde G$, and these subsets are compact
  by Theorem~\ref{thm:4.x}. 
\end{prf}

\subsection{The projection to $K$}

In this subsection we verify a result on the
  Guichardet--Wigner quasimorphism $f_{GW}$
  (Proposition~\ref{prop:timefun})
  that will be used in the following section. Here the key point
is that $f_{GW}$ behaves nicely on $G^{se}(0)$.

From Proposition~\ref{prop:2.11} we recall that the map
\[ \phi \: \fp\times K^{se} \to G^{se}, \quad
  (x,k) \mapsto \exp (x) k \exp(-x) \]
is a diffeomorphism. Accordingly, we obtain a smooth map
\[ \psi \: G^{se} \to K^{se}, \quad
  \phi(x,k) \mapsto k.\]

\begin{lem} \mlabel{lem:kproj} The map $\psi$ is monotone with respect to the
  biinvariant cone fields defined by $W$ on $G^{se}$
  and $W_\fk := W \cap \fk$ on $K^{se}$,   respectively.
\end{lem}

\begin{prf} For $p := \exp(x)$, $x \in \fp$, let
  $c_p(g) := pgp^{-1}$. Then
  $\psi \circ c_p = \psi$ and the cone field on $G^{se}$ is invariant
  under all these maps~$c_p$. We therefore
  have
  \begin{align*}
 T_{c_p(k)}(\psi)(c_p(k).W)
&= T_{c_p(k)}(\psi)(pk.W.p^{-1}) 
= T_{c_p(k)}(\psi) T_k(c_p)(k.W) \\
&= T_k(\psi \circ c_p)(k.W) = T_k(\psi)(k.W).
\end{align*}

The tangent map $T_k(\psi)$ maps $T_k(G) = k.\g = k.\fk + k.\fp$ to
$T_k(K) = k.\fk$ and $\psi\res_{K^{se}} = \id_{K^{se}}$, so that
$T_k(\psi)$ is a projection onto $T_k(K)$.
Its kernel is the tangent space of the submanifold
$\bigcup_{p \in \exp\fp} pkp^{-1}$.
As $pkp^{-1} = k(k^{-1}pkp^{-1})$
and the map 
\[x \mapsto  k^{-1} \exp(x) k \exp(-x) = \exp(\Ad(k)^{-1}x) \exp(-x) \]
produces in $0$ the tangent vectors
$\Ad(k)^{-1}x  - x$ for $x \in \fp$. This is the linear
subspace \break $(\Ad(k)^{-1}-\1)\fp = \fp$,
where we use $k \in K^{se}$ for this equality. This implies that
\[ T_k(\psi) \: k.\g \to k.\fk \]
is the linear projection along the subspace $k.\fp$.

We therefore obtain 
\[  T_{c_p(k)}(\psi)(c_p(k).W)
  =T_k(\psi)(k.W) = k.p_\fk(W)\ {\buildrel \ref{lem:Wk}\over =}\ k.W_\fk.\qedhere\]
\end{prf}

\begin{lem} \mlabel{lem:Wk} $p_\fk(W) = W_\fk = W \cap \fk$. 
\end{lem}

\begin{prf} Clearly, $W_\fk \subeq p_\fk(W)$. For the converse,
  we observe that, for $w \in W$, we have $p_\ft(w) = p_\ft(p_\fk(w))$,
  so that 
  \[ p_\ft(\Ad(K)p_\fk(w)) \subeq p_\ft(\Ad(K)w)) \subeq
    p_\ft(W) = W_\ft := W \cap \ft.\]
  Hence
  \[  p_\fk(W) \subeq \bigcap_{k \in K} \Ad(k).p_\ft^{-1}(W_\ft)
    = W_\fk,\]
  where the last equality follows from the fact that
  both sides are invariant cones in $\fk$ with the same intersection
  with $\ft$ (\cite[Prop.~VII.3.5]{Ne99}).
\end{prf}

\begin{prop} \mlabel{prop:timefun} Suppose that $G$ is simply connected.
  Then $f_{GW}\res_{G^{se}(0)}$ is a bounded smooth $W$-monotone function on $G^{se}(0)$.
\end{prop}

\begin{prf}   As $f := f_{GW}\res_{G^{se}(0)}$ is conjugation invariant
  and $f\res_K = p_Z$,   we have $f = p_Z \circ \psi$. The function
\[ \psi \:  G^{se}(0) \to K^{se}(0) \] 
  is monotone by Lemma~\ref{lem:kproj}, and $p_Z \: K \to \fz(\fk)$
  is a group homomorphism whose differential
  $p_{\fz(\fk)}$ maps $W_\fk = p_{\fz}(W)$ to $W \cap \fz(\fk)$. 
  Therefore $p_Z$ is also monotone with respect to the causal structure
  defined by $W_\fk$ on~$K$. In view of Lemma~\ref{lem:Wk},
  this implies the monotonicity of~$f$.
\end{prf}

\section{The global hyperbolicity of $G^{se}(0)$
in the general case}
\mlabel{sec:6}

In this section we show that $G^{se}(0)$ is also globally hyperbolic
if the hermitian Lie algebra $\g$ is not of tube type.
Then the subsemigroup $S_W \subeq G$ coincides with the
whole group, so that we cannot refer to the global
order structure on $G$ corresponding to the biinvariant
cone field $(g.W)_{g \in G}$. 


Consider an invariant cone $W$ in a Lie-algebra $\mathfrak{g}$.
The invariant cone defines a bi-invariant causal structure on the Lie-group $G$ by $W_g:=g\cdot W$ for the canonical left action of $G$ on $T(G)$.

We call a smooth curve $\gamma \: I \to G$ \textit{timelike} 
if $\gamma'(t)\in W_{\gamma(t)}^\circ$ for $t \in I$ and
\textit{causal} if $\gamma'(t)\in W_{\gamma(t)}$
for every $t \in I$.

The \textit{chronological future and past} of $g\in G$ can be defined as
\begin{align*}
I^+(g):=\{h\in G|\text{ there exists a timelike curve in $G$ from $g$ to $h$}\}\\
I^-(g):=\{h\in G|\text{ there exists a timelike curve in $G$ from $h$ to $g$}\}
\end{align*}
and the \textit{causal future and past} as
\begin{align*}
J^+(g):=\{h\in G|\text{ there exists a causal curve in $G$ from $g$ to $h$}\}\\
J^-(g):=\{h\in G|\text{ there exists a causal curve in $G$ from $h$ to $g$}\}.
\end{align*}

Note that, due to the invariance of the causal structure,
the future and past sets are given by $I^{\pm}(g)=g I^{\pm}(e)=I^{\pm}(e)g$ and $J^{\pm}(g)=g J^{\pm}(e)=J^{\pm}(e)g$.
Moreover $I^-(e)=(I^+(e))^{-1}$ and $J^-(e)=(J^+(e))^{-1}$, i.e., the causal structure is entirely determined by the chronological/causal future of $e$.

In the following we consider the causal/chronological future and past of the causal structure restricted to $G^{se}(0)$, i.e.,  we restrict to causal curves that are contained in $G^{se}(0)$.
Recall that $G^{se}(0)$ is globally hyperbolic if it does not contain causal loops and the order intervals $J^+(g)\cap J^-(h)$ are compact for all $g,h\in G^{se}(0)$. 
To show global hyperbolicity, it makes no difference if we consider the order intervals induced by $J^{\pm}(g)$ or the ones induced by $\overline{J^{\pm}(g)}$, since the compactness of the two is known to be equivalent
(assuming that $W$ has non-empty interior,
see \cite[Proposition 2.20]{Mi19}.
The global hyperbolicity of a closed causal structure has strong consequences on the topology of the underlying manifold and the causal properties of the causal structure.
For instance it implies the existence of a smooth splitting $G^{se}(0)\cong \mathbb{R}\times \Sigma$, where $\Sigma$ a Cauchy hyperspace, i.e., a hyperspace in $G^{se}(0)$ that is intersected in a unique point by every 
inextendible  causal curve (contained in $G^{se}(0)$), see e.g. \cite{BS18, Mi19} for further discussion on the properties of globally hyperbolic causal structures.

\begin{thm}\label{thm1} 
  For the bi-invariant causal structure on $G$, defined by $W_g=g\cdot W$
  for the canonical left action of $G$ on $T(G)$,
  the open subset $G^{se}(0)\subset G$ is globally hyperbolic.
\end{thm}

\subsection*{Proof of the theorem}

Let $\Delta_p^+$ be as before.
Define 
$$\tau\colon G^{se}(0)\rightarrow \mathbb{R}$$
as the unique conjugation invariant function that is defined on
$\exp(\mathfrak{t}^{se}(0))$ by 
$$ \tau(\exp x) := \sum\limits_{\alpha\in \Delta_p^+ }\ln \underbrace{(i\alpha(x))}_{> 0}-\ln(\underbrace{2\pi-i\alpha(x)}_{> 0}).$$

\begin{lem}\label{lem1}
The function $\tau$ is strictly increasing along smooth causal curves.
\end{lem}

\begin{proof}
Using Lemma 4.20 it remains to show that the function
\begin{align*}
  \phi\colon \mathfrak{t}^{se}(0)&\rightarrow \mathbb{R}, \quad
\phi(x) := \sum\limits_{\alpha\in \Delta_p^+ }\ln (i\alpha(x))-\ln(2\pi-i\alpha(x))
\end{align*}
satisfies 
$d\phi(x)(W_{\mathfrak{t}}^{\max})\subset (0,\infty)$. 
One has 
$$d\phi(x)(y)=\sum\limits_{\alpha\in \Delta_p^+ }\left(\frac{1}{i\alpha(x)}+\frac{1}{2\pi-i\alpha(x)}\right)i\alpha(y).$$
The claim follows since for $0 \not=y\in W_{\mathfrak{t}}^{\max}$
there exists $\alpha\in \Delta_p^+$ with $\alpha(y)>0$.
\end{proof}

The following lemma shows that, for $g \in G^{se}(0)$,
any timelike curve of the form $\exp(tx)g$ exits $G^{se}(0)$ in finite time in the future and past.

\begin{lem}\label{lem2}
  For any non-nilpotent $x\in W$
  and $g\in G^{se}(0)$, there exist $c_1 < 0 < c_2$ 
  such that $\exp(tx)g\in G^{se}(0)$ for all $t\in (c_1,c_2)$ and $\exp(c_j x)g
  \notin G^{se}(0)$. In particular
\[   \lim\limits_{t\searrow c_1}\tau(\exp(tx)g)=-\infty
                               \quad \mbox{ and } \quad 
\lim\limits_{t\nearrow c_2}\tau(\exp(tx)g)=\infty.\] 
\end{lem} 

\begin{proof} 
  By \cite[Cor.~B.2]{NOe22} each $x\in W$ has a
  Jordan decomposition $x=x_e+x_n$, where $x_e\in W$ is elliptic and $x_n\in W$ nilpotent.
Since $x \not= x_n$ by assumption, we have $x_e\neq 0$.
  
It follows that $x_e \in \mathrm{Ad}(G)\mathfrak{t}\cap W$.
Pick $g\in G$ with $\Ad(g)x_e \in \ft$.
Then there exists $\alpha\in \Delta_p^+$ with $\alpha(\mathrm{Ad}(g)x_e)>0$.
It follows from Proposition~\ref{prop:3.11} that the universal covering map $q_G \: \tilde G \to G$ maps $\tilde G^{se}(0)$ diffeomorphically to   $G^{se}(0)$. 
Recall the quasimorphism $f_{GW} \: \tilde G \to Z(\tilde K)_e \cong \R$.
By \cite[Prop.~3.3.4]{Ha10} (cf.\ also Subsection~\ref{subsec:GW}), on $\exp \left(\mathrm{Ad}(G)\mathfrak{t}\right)$, this quasimorphism is given by 
\[ f_{GW}\left(\exp\left(\mathrm{Ad}(g)y\right)\right)
  =\frac{1}{|\Delta_p^+|}\sum\limits_{\alpha\in\Delta_p^+}i\alpha(y)
  \quad \mbox{ for } \quad y \in \ft.\]

Note that $f_{GW}$ is bounded on $G^{se}(0)$ since there it coincides with the function $f$ from Proposition~\ref{prop:timefun}.
On the other hand $f_{GW}(\exp(t \Ad(g)x_e))$ is by the above formula unbounded from above and below and increasing in $t$.
The quasimorphism property and the fact that $f_{GW}(\exp(tx_n))=0$ imply that $f_{GW}(\exp(tx)g)=f_{GW}(\exp(tx_e)\exp(tx_n)g)$ is unbounded from above and from below in~$t$.
Therefore $\exp(tx)g$ leaves $G^{se}(0)$ for finite negative and finite positive $t$.
It follows from Theorem~\ref{properness:theorem} that
$\psi(\exp(tx)g)$ leaves every
compact subset in $K^{se}(0)$ for finite positive and negative $t$.
As~$K^{se}(0)$ has compact closure in $K$, Theorem \ref{thm:conncomp} implies that
$\tau(\psi(\exp(tx)g))$ and hence $\tau(\exp(tx)g)$ diverge.
\end{proof}

\begin{lem}\label{lem3}
Let $g,h\in G^{se}(0)$ with $h\in I^+(g)$.
Then there exists $x\in W$ with $\exp(x)g=h$ and $\exp(sx)g
\in G^{se}(0)$ for all $s\in [0,1]$.
\end{lem}

\begin{proof} (a) We first show that $hg^{-1}\in G^{se}(0)$. 
Suppose $hg^{-1}\notin G^{se}(0)$. 
This will lead to a contradiction to the properties of the time function $\tau$.

Let $\gamma(s)$ be a timelike curve in $G^{se}(0)$ with $\gamma(0)=g$ and $\gamma(1)=h$.
Choose $y\in \mathfrak{g}^{se}(0)$ with $\exp(y)=g$.
Then $\exp(ty)\in G^{se}(0)$ for all $t\in (0,1]$. 

We consider the continuous map 
\[ H\colon [0,1]\times [0,1]\rightarrow G, \quad 
H(s,t) := H^t(s) := H_s(t) := \gamma(s) \exp(-ty).\] 
Then
\[ H(s,0)=\gamma(s) \quad \mbox{ and } \quad H(1,1)=hg^{-1} \not\in G^{se}(0).\] 
Since $G^{se}(0)$ is open, the set
\[  I := \{ t \in [0,1] \: H^t([0,1]) \subeq G^{se}(0) \} \]
contains $0$ and is open, so that 
\[ t_0:=\min([0,1] \setminus I) \in (0,1].\]
From $H(0,t_0)=\exp((1-t_0)y)\in G^{se}(0)\cup\{e\}$, we derive that
\[ s_0 := \min \{s \in (0,1] \: H(s,t_0)\not\in G^{se}(0) \} > 0.\]
In the case that $t_0=1$, we use the fact
that $H^{t_0}(s)$ is timelike and any timelike curve starting at $e$ immediately enters $G^{se}(0)$. 

By Lemma \ref{lem2} the function $\tau$ goes to $-\infty$
 for $t \to t_0-$ 
along $H_{s_0}(t)$ and is increasing along $H^{t_0}(s)$.
Thus we can pick $\epsilon> 0$ small enough such that $\tau(H(s_0,t_0-\epsilon))< \tau (H(s_0-\epsilon,t_0))$.
On the other hand $H(s_0-\epsilon,t_0-\epsilon)\in I^+(H(s_0-\epsilon,t_0))$ and $H(s_0,t_0-\epsilon)\in I^+(H(s_0-\epsilon,t_0-\epsilon))$.
In particular $H(s_0,t_0-\epsilon)\in I^+(H(s_0-\epsilon,t_0))$.
Here we use the minimality of $t_0,s_0$ and the causality properties of $H^t,H_s$.
 
This yields the desired contradiction to $\tau$ being a time function
and we conclude that $hg^{-1} \in G^{se}(0)$. 

\nin (b) Choose $x\in\mathfrak{g}^{se}(0)$ with $\exp(x)=hg^{-1}$.
Then $\exp(x)g=h$.
For $s\in [0,1)$ we have that 
$$\exp(sx)g=\exp((s-1)x)\exp(x)g=\exp((s-1)x)h.$$
Since $h\in I^+(\exp((1-s)x))$ for all $s\in [0,1)$,
part (a) yields  $\exp(sx)g\in G^{se}(0)$ for $s \in [0,1]$.
\end{proof}

We are now ready to prove Lemma \ref{lem2} also for nilpotent elements
  $x\in W$. 

\begin{lem}\label{lem2b}
  For any $x\in W$ and $g\in G^{se}(0)$ there exist $c_1 < 0 < c_2$ 
  such that $\exp(tx)g\in G^{se}(0)$ for all $t\in (c_1,c_2)$
  and $\exp(c_j x)g
  \notin G^{se}(0)$. In particular
\[   \lim\limits_{t\searrow c_1}\tau(\exp(tx)g)=-\infty
                               \quad \mbox{ and } \quad 
\lim\limits_{t\nearrow c_2}\tau(\exp(tx)g)=\infty.\] 
\end{lem} 

\begin{proof} 
As in the proof of Lemma \ref{lem2} we look at the Jordan decomposition $x=x_n+x_e$.
It remains to deal with the case where $x=x_n$ is nilpotent.
We show the existence of $0<c_2<\infty$, the existence of $c_1$ can be shown analogously.

Let $\epsilon>0$ be small enough such that $g\in I^+(\exp(\epsilon z))$.
It follows from Lemma \ref{lem3} that the curve $\exp(tx)\exp(\epsilon z)$ is contained in $G^{se}(0)$ as long as $\exp(tx)g$, i.e., we can assume that $g=\exp(\epsilon z)$.

The classification of nilpotent elements in $W$
(see Theorem~III.9 and its proof in \cite{HNO94}) implies the existence of a
subalgebra $\fl\subset\fg$ with $z\in \fl$
and $\fl=\fz(\fl)\oplus[\fl,\fl]\cong \fz(\fl)\oplus \fsl_2(\R)^{\oplus r}$
and some $k \in K$ with $\Ad(k)x \in \fl$.
  As $\Ad(k)$  fixes $\exp(\epsilon z)$, we
can assume that $x\in [\fl,\fl]\cong (\fsl_2(\R))^r$.
Writing $z=z_1+z_2$, where $z_1\in \fz(\fl)$ and $z_2\in [\fl,\fl]$, the curve $\exp(tx)\exp(\epsilon z)=\exp(tx)\exp(\epsilon z_2)\exp(\epsilon z_1)$ is stably elliptic in $L := \la \exp\fl \ra$ 
if and only if $\exp(tx)\exp(\epsilon z_2)$ is stably elliptic.
As $G^{se}(0)$ is the same in locally isomorphic groups (Proposition~\ref{prop:3.11})
and element in $L$ that are not stably elliptic in $L$ are
  not stably elliptic in $G$,  it suffices to prove the lemma for the 
group $\SL_2(\R)^r$. As the set of stably elliptic elements adapts
  to the product structure, we may further assume $r = 1$.
We can therefore assume that $\exp(tx)\exp(\epsilon z)\subset \SL_2(\R)$.
As $z$ is elliptic, up to conjugation and scaling, we may assume that 
\[ z = \pmat{0 & -1 \\ 1 & 0},\quad
  x = \pmat{ 0 & -1 \\ 0 & 0},\]
so that
\[  \exp(tx)\exp(\eps z)
  = \pmat{ 1 & -t \\ 0 & 1} \pmat{ \cos \eps & -\sin \eps \\ \sin \eps & \cos \eps}
  =  \pmat{ \cos \eps-t\sin \eps &  -\sin \eps -t\cos \eps \\
    \sin \eps &  \cos \eps}.\]
The trace of an elliptic element $g \in \SL_2(\R)$
with eigenvalues $\cos(\theta) \pm i \sin (\theta)$ is
$2 \cos(\theta) \in [-2,2]$, but
\[ \tr(\exp(tx)\exp(\epsilon z))  = 2 \cos \eps - t \sin \eps \]
tends to $-\infty$ for $t \to \infty$.
Hence the assertion holds in $\SL_2(\R)$.
The divergence of $\tau$ along $\exp(tx)\exp(\eps z)$
follows analogously to Lemma~\ref{lem2}. 
\end{proof}

For a different proof of Lemma~\ref{lem2b} for the group
$\SL_2(\R)\cong \Sp_2(\R)$ we refer to \cite{He22}.

Lemma \ref{lem2b} allows to prove Lemma \ref{lem3} for every $h\in J^+(g)$.

\begin{lem}\label{lem3b}
Let $g,h\in G^{se}(0)$ with $h\in \overline{J^+(g)}$.
Then there exists $x\in W$ with $\exp(x)g=h$ and $\exp(sx)g\in G^{se}(0)$ for all $s\in [0,1]$.
In particular $J^+(g)$ is closed in $G^{se}(0)$
for all $g$.
\end{lem}

\begin{proof}

In view of Lemma~\ref{lem3}, it remains to consider the case $h\in \overline{J^+(g)}\setminus I^+(g)$.

Since the causal structure on $G$ induced by the invariant cone $W$ is locally Lipschitz in the sense of \cite{FS11, Mi19}, \cite[Thm.~2.7]{Mi19} implies
that $\overline{J^+(g)}= \overline{I^+(g)}$.

Thus we can pick a sequence $h_n\in I^+(g)$ with $h_n\rightarrow h$.
Let $D$ be the intersection of the cone $W\subset \mathfrak{g}$
with the unit sphere of some norm on $\g$. 
Using Lemma \ref{lem3} we can find a sequence $x_n$ in $D$ and a sequence $t_n$ such that $h_n=\exp(t_nx_n)g$.

By the compactness of $D$, passing to a subsequence, we may
  assume that $x_n \rightarrow x \in D$. Consider the curve $\eta(t):=\exp(tx)g$.
Now suppose that there is no $t$ that satisfies the assumptions of the
lemma, i.e., there is no $t$ such that $\exp(tx)g=h$ and such that $\exp(stx)g\in G^{se}(0)$ for all $s\in [0,1]$.
Lemma \ref{lem2b} implies that there exists some $t_0$ such that $\eta(t)\in G^{se}(0)$ for $t\in [0,t_0)$ and $\eta(t_0)\notin G^{se}(0)$.
Moreover we have $\lim\limits_{t\rightarrow t_0-}\tau(\eta(t))=\infty$.
In particular there exists
$t_c\in (0,t_0)$ such that $\tau(\eta(t_c))=c>\tau(h)$.

On the other hand the limit curve theorem for causal curves in locally Lipschitz causal structures \cite[Thm.~2.14]{Mi19} implies that $\eta(t)\subset \overline{J^-(h)}$ for all $t<t_0$, so that 
$\tau(\eta(t))\leq \tau(h)$ for all $t\in [0,t_0)$.

This contradicts $\tau$ being a time function on $G^{se}(0)$ as proved in Lemma \ref{lem1}.
Hence there exists a $t$ such that $tx$ satisfies the assumptions of the
lemma. 
\end{proof}

\begin{proof}[Proof of Theorem~\ref{thm1}]
  We follow the strategy in \cite{He22}, which treats the special case
  $G=\Sp_{2n}(\R)$. We have
  to show that for all $g,h\in G^{se}(0)$, $J^+(g)\cap J^-(h)$ is compact.

Let $g_k$ be a sequence in $J^+(g)\cap J^-(h)$.
As before, let $D$ be the intersection of $W$ with the unit sphere of some
norm on $\g$.
Applying Lemma \ref{lem3b} we can pick a sequence $x_k$ in $D$ and a sequence $t_k$ such that $g_k=\exp(t_kx_k)g$.
Since $D$ is compact, $x_k$ converges up to a subsequence to some $x\in D$.
By Lemma \ref{lem2b} we can pick a $c$ such that $\exp(tx)g\in G^{se}(0)$ for all $t\in[0,c)$ and $\exp(cx)g\notin G^{se}(0)$.
Note that $\tau(g_k)\leq \tau(h)$ for all $k$.

We claim that $\lim\sup t_k < c$.
  Otherwise, passing to a subsequence, may assume
  that $t_k \to c$, and this would imply that
  $\tau(\exp(tx)g)$ is bounded for $t \in [0,c)$; a contradiction.
Hence a subsequence of $(t_k)_{k \in \N}$
converges to some $t_{{\rm max}}\in [0,c)$ and in particular a
subsequence of $(g_k)_{k \in \N}$
converges to $\exp(t_{{\rm max}}x)g\in J^+(g)$.
Without loss of generality, we
assume that the sequence $(g_k)_{k \in \N}$ converges to $\exp(t_{{\rm max}}x)g$.

Using Lemma \ref{lem3b}, there exist sequences $y_k$ in $D$ and $s_k$ such that
$g_k=\exp(-s_ky_k)h$.
Repeating the above argument, we obtain that $g_k$ converges to
$\exp(t_{{\rm max}}x)g=\exp(-s_{{\rm max}}x)h
\in J^+(g)\cap J^-(h)$, which implies the desired compactness.
\end{proof}

\section{Perspectives}
\mlabel{sec:7}

\begin{prob}
It would be interesting to explore causal properties of the whole group $G$ in the case when there are no closed causal curves.
For example in \cite{ABP22} the authors construct a natural invariant
Lorentz--Finales metric on $\Sp_{2n}(\R)$ and its universal cover, as well as a family of time functions on $\tilde\Sp_{2n}(\R)$
using the invariant quasimorphism (might be possible on further groups using the GW quasimorphism?).
\end{prob}

\begin{prob} For which groups $G$ and which cones $W$, is
the   corresponding left invariant causal structure $(g.W)_{g \in G}$
  {\it causally simple} in the sense that $J^+(e)$ is closed,
  hence equal to $\overline{\langle \mathrm{exp}(W)\rangle}$? 
\end{prob}

\begin{prob} Suppose that $\g$ is of tube type. Is it true
  that
  \[ \partial S_W = \exp(\partial W)?\]
  According to \cite{HHL89}, this is true for $G = \tilde \SL_2(\R)$.

  Are the order interval
  \[ [e,s] = S_W \cap  s S_W^{-1} \]
  compact for every $s \in \partial S_W$? 
\end{prob}

\end{document}